\documentclass[12pt]{amsart}

\usepackage{amsmath}
\usepackage{amsfonts}
\usepackage{amssymb}
\usepackage{latexsym}

\newtheorem{theorem}{Theorem}[section]
\newtheorem{proposition}[theorem]{Proposition}
\newtheorem{corollary}[theorem]{Corollary}

\theoremstyle{remark}
\newtheorem{remark}[theorem]{Remark}
\theoremstyle{definition}
\newtheorem{example}[theorem]{Example}
\numberwithin{equation}{section}

\def\sD{{\mathfrak D}}      
   \def\sH{{\mathfrak H}}   
   \def\sK{{\mathfrak K}}   \def\sL{{\mathfrak L}}
\def\sM{{\mathfrak M}}   \def\sN{{\mathfrak N}}

      \def\dC{{\mathbb C}}
\def\dD{{\mathbb D}}

   \def\dT{{\mathbb T}}

   \def\cH{{\mathcal H}}   
   \def\cK{{\mathcal K}}   \def\cL{{\mathcal L}}
\def\cM{{\mathcal M}}

\def\cV{{\mathcal V}}

\def\bL{{\mathbf L}}
\def\RE{{\rm Re\,}}

\def\wt{\widetilde}
\def\wh{\widehat}

\def\f{\varphi}

\def\uphar{{\upharpoonright\,}}

\def\ran{{\rm ran\,}}
\def\cran{{\rm \overline{ran}\,}}
\def\ker{{\rm ker\,}}
\def\dom{{\rm dom\,}}

\def\dim{{\rm dim\,}}

\def\cspan{{\rm \overline{span}\, }}

\keywords{Passive system, transfer function, shorted operator,
Kalman - Yakubovich - Popov inequality, Riccati equation}
\subjclass[2000]{Primary: 47A48, 47A56, 47A63, 47A64, 93B28;
Secondary 93B15, 94C05.}

\begin{document}
\title [The Kalman--Yakubovich--Popov inequality]
{The Kalman--Yakubovich--Popov inequality for passive discrete
time-invariant systems}
\author{
Yury~Arlinski\u{i}}
\address{Department of Mathematical Analysis \\
East Ukrainian National University \\
Kvartal Molodyozhny 20-A \\
Lugansk 91034 \\
Ukraine} \email{yury\_arlinskii@yahoo.com}

\begin{abstract}

We consider the Kalman - Yakubovich - Popov (KYP)
 inequality
 \[
\begin{pmatrix} X-A^* XA-C^*C &
-A^*X B- C^*D\cr -B^*X A-D^* C & I- B^*X B-D^*D
\end{pmatrix}
\ge 0
 \]
  for contractive operator matrices
$
\begin{pmatrix} A&B\cr C &D
\end{pmatrix}:\begin{pmatrix}\mathfrak{H}\cr\mathfrak{M} \end{pmatrix}\to\begin{pmatrix}\mathfrak{H}\cr\mathfrak{N}
\end{pmatrix},
$ where $\mathfrak{H},$ $\mathfrak{M}$, and $\mathfrak{N}$ are
separable Hilbert spaces.
 We restrict ourselves to the solutions $X$
from the operator interval $[0, I_\mathfrak{H}]$. Several equivalent
forms of KYP are obtained. Using the parametrization of the blocks
of contractive operator matrices, the Kre\u{\i}n shorted operator,
and the M\"obius representation of the Schur class operator-valued
function we find several equivalent forms of the KYP inequality.
Properties of solutions are established and it is proved that the
minimal solution of the KYP inequality satisfies the corresponding
algebraic Riccati equation and can be obtained by the iterative
procedure with the special choice of the initial point. In terms of
the Kre\u{\i}n shorted operators a necessary condition and some
sufficient conditions for uniqueness of the solution are
established.

\end{abstract} \maketitle

\tableofcontents
\section{Introduction}
The system of equations
\[
\left\{
\begin{split}
&h_{k+1}=Ah_k+B\xi_k,\\
&\sigma_k=Ch_k+D\xi_k
\end{split}
\right.,\qquad k\ge 0
\]
describes the evolution of a \textit{linear discrete time-invariant
system} $\tau=\left\{\begin{pmatrix} A&B \cr
C&D\end{pmatrix};\sH,\sM,\sN\right\}$ with bounded linear operators
$A$, $B$, $C$, $D$ and separable Hilbert spaces $\sH$ (state space),
$\sN$ (input space), and $\sM$ (output space) $\sM$. If the linear
operator $T_\tau$ operator by the block-matrix
\[
T_\tau=\begin{pmatrix} A&B \cr C&D\end{pmatrix} : \begin{pmatrix}
\sH \\ \sM \end{pmatrix} \to
\begin{pmatrix} \sH \\ \sN \end{pmatrix}
\]
is contractive, then the corresponding discrete-time system is said
to be \textit{passive}. If the block-operator matrix $T_\tau$ is
isometric (co-isometric, unitary) then the corresponding system is
called isometric (co-isometric, conservative). Isometric and
co-isometric systems were studied by L.~de Branges and J.~Rovnyak
\cite{BrR1}, \cite{BrR2} and by T.~Ando \cite{Ando}, conservative
systems have been investigated by B.~Sz.-Nagy and C.~Foias \cite{SF}
and M.S.~Brodski\u{\i} \cite{Br1}. Passive systems are studied by
D.Z.~Arov et al \cite{A, Arov, ArKaaP, ArKaaP3, ArNu1, ArNu2, ArSt}.
The subspaces
\begin{equation}
\label{CO} \sH^c_\tau:=\cspan\{A^{n}B\sM:\,n=0,1,\ldots\}
\quad\mbox{and}
\quad\sH^o_\tau=\cspan\{A^{*n}C^*\sN:\,n=0,1,\ldots\}
\end{equation}
 are called the
\textit{controllable} and \textit{observable} subspaces of the
system $\tau=\left\{\begin{pmatrix} A&B \cr
C&D\end{pmatrix};\sH,\sM,\sN\right\}$, respectively. If
$\sH^c_\tau=\sH$ ($\sH^o_\tau=\sH$) then the system $\tau$ is said
to be \textit{controllable} (\textit{observable}), and
\textit{minimal} if $\tau$ is both controllable and observable. If
$\sH=\rm{closure}\{\sH^c_\tau+\sH^o_\tau\}$ then the system $\tau$
is said to be a \textit{simple}. Note that from \eqref{CO} it
follows that
\[
(\sH^c_\tau)^\perp=\bigcap\limits_{n=0}^\infty\ker(B^*A^{*n}),\;
(\sH^o_\tau)^\perp=\bigcap\limits_{n=0}^\infty\ker(CA^{n}).
\]
Therefore
\begin{enumerate}
\item the system $\tau$ is controllable $\iff
\;\bigcap\limits_{n=0}^\infty\ker(B^*A^{*n})=\{0\}$;
\item the system $\tau$ is observable $\iff
\;\bigcap\limits_{n=0}^\infty\ker(CA^{n})=\{0\}$;
\item the system $\tau$ is simple $\iff
\left(\bigcap\limits_{n=0}^\infty\ker(B^*A^{*n})\right)\bigcap
\left(\bigcap\limits_{n=0}^\infty\ker(CA^{n})\right)=\{0\}.$
\end{enumerate}
The function
\[
\Theta_\tau(\lambda):=D+\lambda C(I_{\sH}-\lambda A)^{-1}B, \quad
\lambda \in \dD,
\]
is called \textit{transfer
function} of the system $\tau$.

The result of D.Z.~Arov~\cite{A} states that two minimal systems
$\tau_1$ and $\tau_2$ with the same transfer function
$\Theta(\lambda)$ are  \textit{pseudo-similar}, i.e. there is a
closed densely defined operator $Z:\sH_{1}\to\sH_{2}$ such that $Z$
is invertible, $Z^{-1}$ is densely defined, and
\[
Z A_1f =A_2 Zf, \; C_1f=C_2 Zf,\; f\in\dom Z,\;\mbox{and} \quad
ZB_1=B_2.
\]
If the system $\tau$ is passive then $\Theta_\tau$ belongs to the
\textit{Schur class} ${\bf S}(\sM,\sN)$, i.e.,
$\Theta_\tau(\lambda)$ is holomorphic in the unit disk
$\dD=\{\lambda \in \dC:|\lambda|<1\}$ and its values are contractive
linear operators from $\sM$ into $\sN$. It is well known
\cite{BrR2}, \cite{SF}, \cite{Ando}, \cite{A}, \cite{ArKaaP} that
every operator-valued function $\Theta(\lambda)$ of the Schur class
${\bf S}(\sM,\sN)$ can be realized as the transfer function of some
passive system, which can be chosen as a simple conservative
(isometric controllable, co-isometric observable, passive and
minimal, respectively). Moreover, two simple conservative (isometric
controllable, co-isometric observable) systems
$$\tau_1=\left\{\begin{pmatrix} A_1&B_1 \cr
C_1&D\end{pmatrix};\sH_1,\sM,\sN\right\}\quad\mbox{and}\quad
\tau_2=\left\{\begin{pmatrix} A_2&B_2 \cr
C_2&D\end{pmatrix};\sH_2,\sM,\sN\right\}$$
 having the same transfer
function are unitarily similar \cite{BrR1}, \cite{BrR2}, \cite{Br1},
\cite{Ando}, i.e. there exists a unitary operator $U$ from $\sH_{1}$
onto $\sH_{2}$ such that
\[
A_1 =U^{-1}A_2U,\; B_1=U^{-1}B_2,\; C_1=C_2 U.
\]

In \cite{ArNu1}, \cite{ArNu2} necessary and sufficient conditions on
operator-valued function $\Theta(\lambda)$ from the Schur class
${\bf S}(\sM,\sN)$ have been established in order that all minimal
passive systems having the transfer function $\Theta(\lambda)$ be
unitarily similar or similar.

A system $\tau=\left\{\begin{pmatrix} A&B \cr
C&D\end{pmatrix};\sH,\sM,\sN\right\}$ is called \textit{$X$-passive}
with respect to the supply rate function
$w(u,v)=||u||^2_\sM-||v||^2_\sN$, $u\in\sM,$ $v\in\sN$ \cite{Staf}
if there exists a positive selfadjoint operator $X$ in $\sH$,
possibly unbounded, such that
\[
A\,\dom X^{1/2}\subset \dom X^{1/2},\;\ran B\subset\dom X^{1/2}
\]
and
\begin{equation}
\label{dissip}
 ||X^{1/2}(Ax+Bu)||^2_\sH-||X^{1/2}x||^2_\sH\le
||u||^2_\sM-||Cx+Du||^2_\sN,\;x\in\dom X^{1/2},\,u\in\sM.
\end{equation}
The condition \eqref{dissip} is equivalent to
\begin{equation}
\label{kyp3}
\begin{split}
&\left\|\begin{pmatrix}{X}^{1/2}&0\cr
0&I_\sM\end{pmatrix}\begin{pmatrix}x\cr u\end{pmatrix}
\right\|^2-\left\|\begin{pmatrix}{X}^{1/2}&0\cr
0&I_\sN\end{pmatrix}\begin{pmatrix} A&B
\cr C&D\end{pmatrix}\begin{pmatrix}x\cr u\end{pmatrix}\right\|^2\ge 0\\
& \qquad\mbox{for all}\quad\; x\in\dom X^{1/2},\, u\in\sM.
\end{split}
\end{equation}
If $X$ is bounded then \eqref{kyp3} becomes to the \textit{Kalman --
Yakubovich -- Popov inequality} (for short, KYP inequality)
\begin{equation}
\label{kyp1} L_{\tau}(X)= \begin{pmatrix} X-A^* XA-C^*C & -A^*X B-
C^*D\cr -B^*X A-D^* C & I- B^*X B-D^*D
\end{pmatrix}
\ge 0.
\end{equation}
The classical Kalman-Yakubovich-Popov lemma states that \textit{if
$\tau=\left\{\begin{pmatrix} A&B \cr
C&D\end{pmatrix};\sH,\sM,\sN\right\}$ is a minimal system with
finite dimensional state space $\sH$ then the set of the solutions
of \eqref{kyp1} is non-empty if and only if the transfer function
$\Theta_\tau$ belongs to the Schur class. If this is a case then the
set of all solutions of \eqref{kyp1} contains the minimal and
maximal elements}.

For the case $\dim\sH=\infty$ the theory of generalized KYP
inequality \eqref{kyp3} is developed in \cite{ArKaaP3}.
 It is proved that if $\tau=\left\{\begin{pmatrix} A&B
\cr C&D\end{pmatrix};\sH,\sM,\sN\right\}$ is a minimal system then\\
\textit{the KYP inequality \eqref{kyp3} for $\tau$ has a solution
$X$ if and only if the transfer function $\Theta_\tau$ coincides
with a Schur class function in a neighborhood of the origin.
Moreover, there are maximal $X_{\max}$ and minimal $X_{\min}$
solutions in the sense of quadratic forms:}\\
\textit{if $X$ is a solution of \eqref{kyp3} then $\dom
X^{1/2}_{\min}\supset \dom X^{1/2}\supset \dom X^{1/2}_{\max}$ and}
\[
\begin{split}
& ||X^{1/2}_{\min}u||^2\le ||X^{1/2}u||^2\quad\mbox{for all}\quad
u\in\dom X^{1/2},\\
&||X^{1/2}v||^2\le ||X^{1/2}_{\max}v||^2 \quad\mbox{for all}\quad
v\in\dom X^{1/2}_{\max}.
\end{split}
\]
 Let $\Theta(\lambda)\in{\bf
S}(\sM,\sN)$. A passive system
$$\dot\tau=\left\{\begin{pmatrix} \dot{A}&\dot{B}
\cr \dot{C}&D\end{pmatrix};\dot{\sH},\sM,\sN\right\}$$ with the
transfer function $\Theta(\lambda)$ is called \textit{optimal}
realization of $\Theta(\lambda)$ \cite{Arov}, \cite{ArKaaP} if for
each passive system $\tau=\left\{\begin{pmatrix} A&B \cr
C&D\end{pmatrix};\sH,\sM,\sN\right\}$ with transfer function
$\Theta(\lambda)$ and for each input sequence $ u_0, u_1, u_2,...$
in $\sM$  the inequalities
\[
\left\|\sum\limits_{k=0}^n\dot{A}^k\dot{B}u_k\right\|_{\dot{\sH}}\le
\left\|\sum\limits_{k=0}^n A^kBu_{k}\right\|_{\sH}
\]
 hold for all $n=0,1,...$.

 An observable passive system
\[
\dot\tau_*=\left\{\begin{pmatrix} \dot{A}_*&\dot{B}_* \cr
\dot{C}_*&D\end{pmatrix};\dot{\sH}_{*},\sM,\sN\right\}
\]
is called $(*)$ -\textit{optimal} realization of the function
$\Theta(\lambda)$ \cite{Arov}, \cite{ArKaaP} if for each observable
passive system $\tau=\left\{\begin{pmatrix} A&B \cr
C&D\end{pmatrix};\sH,\sM,\sN\right\}$ with transfer function
$\Theta(\lambda)$ hold
\[
\left\|\sum\limits_{k=0}^n\dot{A}^k_*\dot{B}_*u_{k}\right\|_{\dot{\sH}_{*}}\ge
\left\|\sum\limits_{k=0}^n A^kBu_{k}\right\|_\sH
\]
for all $n=0,1,...$ and every choice of vectors $u_0,u_1, u_2,...,
u_n\in\sM$.

Two minimal and optimal ($(*)$-optimal) passive realizations of a
function from the Schur class are unitary similar \cite{ArKaaP}. In
addition the system
\[
\dot\tau_*=\left\{\begin{pmatrix} \dot{A}_*&\dot{B}_* \cr
\dot{C}_*&D\end{pmatrix};\dot{\sH}_{*},\sM,\sN\right\}
\]
is $(*)$-optimal passive minimal realization of the function
$\Theta(\lambda)$ if and only if the system
\[
\dot\tau_*^*=\left\{\begin{pmatrix} \dot{A}^*_*&\dot{B}^*_* \cr
\dot{C}^*_*&D^*\end{pmatrix};\dot{\sH}_{*},\sN,\sM\right\}
\]
is optimal passive minimal realization of the function
$\Theta^*(\overline\lambda)$ \cite{ArKaaP}.

Let $\tau=\left\{\begin{pmatrix} A&B \cr
C&D\end{pmatrix};\sH,\sM,\sN\right\}$ be a minimal system and let
the transfer function $\Theta$ coincides with a Schur class function
in a neighborhood of the origin. Let $X_{\min}$ and $X_{\max}$ be
the minimal and maximal solutions of the KYP inequality
\eqref{kyp3}. It is proved in \cite{ArKaaP} that the systems
\[
\begin{array}{l}
 \dot\tau=\left\{\begin{pmatrix}
X^{1/2}_{\min}AX^{-1/2}_{\min}&X^{1/2}_{\min}B \cr
CX^{-1/2}_{\min}&D\end{pmatrix};\sH,\sM,\sN\right\},\\
 \dot\tau_*=\left\{\begin{pmatrix}
X^{1/2}_{\max}AX^{-1/2}_{\max}&X^{1/2}_{\max}B \cr
CX^{-1/2}_{\max}&D\end{pmatrix};\sH,\sM,\sN\right\}
\end{array}
\]
are minimal optimal and minimal $(*)$-optimal realizations of
$\Theta$, respectively. The contractive operators
$T_1=\begin{pmatrix} X^{1/2}_{\min}AX^{-1/2}_{\min}&X^{1/2}_{\min}B
\cr CX^{-1/2}_{\min}&D\end{pmatrix}$ and $T_2=\begin{pmatrix}
X^{1/2}_{\max}AX^{-1/2}_{\max}&X^{1/2}_{\max}B \cr
CX^{-1/2}_{\max}&D\end{pmatrix}$ are defined on $\dom
X^{-1/2}_{\min}\oplus \sM$ and $\dom X^{-1/2}_{\max}\oplus \sM$,
respectively.

Let $\dT$ be a unit circle. As it is well known \cite{SF} a function
$\Theta(\lambda)$ from the Schur class ${\bf S}(\sM,\sN)$ has almost
everywhere non-tangential strong limit values $\Theta(\xi),$
$\xi\in\dT.$ Denote by $\varphi_\Theta$ ($\psi_\Theta$) the outer
(co-outer) function which is a solution of the following
factorization problem \cite{SF}
\begin{enumerate}
\def\labelenumi{\rm (\roman{enumi})}
\item $\varphi_\Theta^*(\xi)\varphi_\Theta(\xi)\le I_\sM-\Theta^*(\xi)\Theta(\xi)$

\noindent ($\psi_\Theta(\xi)\psi^*_\Theta(\xi)\le
I_\sN-\Theta(\xi)\Theta^*(\xi)$) almost everywhere on $\dT$;
\item if  $\wt\varphi(\lambda)$ ($\wt\psi(\lambda)$) is a Schur class function
such that almost everywhere on $\dT$ holds
$\wt\varphi^*(\xi)\wt\varphi(\xi)\le I_\sM-\Theta^*(\xi)\Theta(\xi)$
($\wt \psi(\xi)\wt\psi^*(\xi)\le I_\sN-\Theta(\xi)\Theta^*(\xi)$)
then

\noindent $\wt\varphi^*(\xi)\wt\varphi(\xi)\le
\varphi_\Theta^*(\xi)\varphi_\Theta(\xi)$ ($\wt
\psi(\xi)\wt\psi^*(\xi)\le \psi_\Theta(\xi)\psi^*_\Theta(\xi)$)
 almost everywhere on
$\dT$.
\end{enumerate}
The function $\varphi_\Theta$ ($\psi_\Theta$) is uniquely defined up
to a left (right) constant unitary factor. The functions
$\f_\Theta(\lambda)$ and $\psi_\Theta(\lambda)$ are called the right
and left \textit{defect functions} associated with $\Theta(\lambda)$
\cite{DR}, \cite{BD}, \cite{BDFK1}, \cite{BDFK2} (in \cite{BC} these
functions are called the right and left spectral factors). By means
of the right (left) defect function the construction of the minimal
and optimal ($(*)$-optimal) realization of the function
$\Theta(\lambda)\in{\bf S}(\sM,\sN)$ is given by D.Z.~Arov in
\cite{Arov}.  In \cite{ArKaaP} the construction of the optimal
($(*)$-optimal) realization is given as the first (second)
restriction of the the simple conservative system with transfer
function $\Theta$.

The next theorem summarizes some results established in \cite{Arov},
\cite{ArKaaP}, \cite{ArNu2}, \cite{BD}, \cite{BDFK1}, \cite{BDFK2}.
\begin{theorem}\label{DBR}
 Let $\Theta(\lambda)\in {\bf
S}(\sM,\sN)$ and let
\[
\tau=\left\{\begin{pmatrix} A&B \cr
C&D\end{pmatrix};\sH,\sM,\sN\right\}
\]
be a simple conservative system with transfer function $\Theta$.
Then
\begin{enumerate}
\item
the subspace $(\sH^o)^\perp$ ($(\sH^c)^\perp$) is invariant with
respect to $A$ ($A^*$) and the restriction $A\uphar (\sH^o)^\perp$
($A^*\uphar (\sH^c)^\perp$) is a unilateral shift;
\item the functions $\varphi_\Theta(\lambda)$ and $\psi_\Theta(\lambda)$
take the form
\begin{equation}
\label{DEFECT}
\begin{split}
& \varphi_\Theta(\lambda)=P_\Omega(I_\sH-\lambda A)^{-1}B,\\
&\psi_\Theta(\lambda)=C(I_\sH-\lambda A)^{-1}\uphar\Omega_*,
\end{split}
\end{equation}
where
\[
\Omega=(\sH^o)^\perp\ominus A(\sH^o)^\perp ,\;
\Omega_*=(\sH^c)^\perp\ominus A^*(\sH^c)^\perp;
\]
\item
$\varphi_\Theta(\lambda)=0$ ($\psi_\Theta(\lambda)=0$) if and only
if the system $\tau$ is observable (controllable).
\end{enumerate}
\end{theorem}
The proof of next theorem is based on the concept and properties of
optimal and $*$-optimal passive systems.
\begin{theorem}
\label{ARNU}\cite{Arov}, \cite{ArKaaP}, \cite{ArNu2}. Let
$\Theta(\lambda)\in{\bf S}(\sM,\sN)$. Then
\begin{enumerate}
\item if $\Theta$ is bi-inner and $\tau$ is a simple passive system
with transfer function $\Theta$ then $\tau$ is conservative system;
\item if $\f_\Theta(\lambda)=0$  or $\psi_\Theta(\lambda)=0$ then
all passive minimal systems with transfer function $\Theta(\lambda)$
are unitary equivalent and if $\f_\Theta(\lambda)=0$ and
$\psi_\Theta(\lambda)=0$ then they are in addition conservative.
\end{enumerate}
\end{theorem}
 In \cite{AHS1} the following
strengthening has been proved using the parametrization of
contractive block-operator matrices.
\begin{theorem}
\label{INNERSYS} Let $\Theta(\lambda)\in{\bf S}(\sM,\sN)$ be a not
constant.
\begin{enumerate}
\item
 Suppose that $\varphi_\Theta(\lambda)=0$,
$\psi_\Theta(\lambda)=0$, and
 $\tau=\left\{\begin{pmatrix}A&B\cr C&D\end{pmatrix};\sH,\sM,\sN\right\}$ is a simple passive
system with transfer function $\Theta$. Then the system $\tau$ is
conservative and minimal.

If  $\Theta (\lambda)$ is bi-inner then in addition the operator $A$
belongs to the class $C_{00}$.
\item
 Suppose that $\varphi_\Theta(\lambda)=0$ ($\psi_\Theta
(\lambda)=0$). If $\tau=\left\{\begin{pmatrix}A&B\cr
C&D\end{pmatrix};\sH,\sM,\sN\right\}$ is a controllable (observable)
passive system with transfer function $\Theta$ then the system
$\tau$ is isometric (co-isometric)and minimal.

If $\Theta (\lambda)$ is inner (co-inner) then in addition the
operator $A$ belongs to the class $C_{0\,\cdot}$ $(C_{\cdot\,0})$.
\end{enumerate}
\end{theorem}
 In this paper we consider the KYP inequality for the case
of contractive operator $T=\begin{pmatrix} A&B \cr
C&D\end{pmatrix}$. Because in this case the set of solutions
contains the identity operator, the minimal solution $X_{\min}$ is a
positive contraction. That's why we are interested in only
contractive positive solutions $X$ of the KYP inequality.

We will keep the following notations. The class of all continuous
linear operators defined on a complex Hilbert space $\sH_1$ and
taking values in a complex Hilbert space $\sH_2$ is denoted by
$\bL(\sH_1,\sH_2)$ and ${\bL}(\sH):= {\bL}(\sH,\sH)$. The domain,
the range, and the null-space of a linear operator $T$ are denoted
by $\dom T,\;\ran T,$ and $\ker T$. For a contraction
$T\in{\bL}(\sH_1,\sH_2)$ the nonnegative square root $D_T = (I -
T^*T)^{1/2}$ is called the defect operator of $T$ and ${\sD}_T$
stands for the closure of the range ${\ran}D_T$. It is well known
that the defect operators satisfy the commutation relation $TD_T =
D_{T^*}T$ and $T\sD_T\subset\sD_{T^*}$ cf. \cite{SF}. The set of all
regular points of a closed operator $T$ is denoted by $\rho(T)$ and
its spectrum by $\sigma(T)$. The identity operator in a Hilbert
space $\cH$ we denote by $I_\cH$ and by $P_\cL$ we denote the
orthogonal projection onto the subspace $\cL$. We essentially use
the following tools.
\begin{enumerate}
\item
The parametrization of the $2\times 2$ contractive block-operator
matrix
$T=\begin{pmatrix} A&B \cr C&D\end{pmatrix} : \begin{pmatrix} \sH \\
\sM \end{pmatrix} \to
\begin{pmatrix} \sK \cr\sN \end{pmatrix}$
\cite{AG}, \cite{DaKaWe}, \cite {ShYa}:
\begin{equation}
\label{BOPR}
\begin{split}
&B=FD_D,\; C=D_{D^*}G,\\
& A=-FD^*G+D_{F^*}LD_{G},
\end{split}
\end{equation}
where the operators $F\in\bL(\sD_D,\sK)$, $G\in\bL(\sH,\sD_{D^*})$
and $L\in\bL(\sD_{G},\sD_{F^*})$ are contractions.
\item The notion of the shorted operator introduced by M.G.~Kre\u{\i}n in
\cite{Kr}:
\[
 S_{\cK}=\max\left\{\,Z\in \bL(\cH):\,
    0\le Z\le S, \, {\ran}Z\subseteq{\cK}\,\right\},
\]
where $S$ is a bounded nonnegative selfadjoint operator in the
Hilbert space $\sH$ and $\cK$ is a subspace in $\cH$.
\item The M\"obius representation
\begin{equation}
\label{MOB} \Theta(\lambda)=\Theta(0)+D_{\Theta^*(0)}
Z(\lambda)(I_{\sD_{\Theta(0)}}+\Theta^*(0)Z(\lambda))^{-1}D_{\Theta(0)},\;\lambda\in\dD
\end{equation}
of the Schur class operator-valued function $\Theta(\lambda)$ by
means of the operator-valued parameter $Z(\lambda)$ from the Schur
class ${\bf S}(\sD_{\Theta(0)}, \sD_{\Theta^*(0)})$.
\end{enumerate}
Such a kind representation was studied in \cite{Shtraus},
\cite{Shmul}, \cite{BC}. The operator $\Gamma_1=Z'(0)$ is called the
\textit{first Schur parameter} and the function
$\Theta_1(\lambda)=\lambda^{-1}Z(\lambda)$ is called the
\textit{first Schur iterate} of the function $\Theta$ \cite{BC}. In
this paper a more simple and algebraic proof of the representation
\eqref{MOB} is given using the equalities \eqref{BOPR}. In
particular, we establish that if $\Theta(\lambda)$ is the transfer
function of the passive system $\tau=\left\{\begin{pmatrix} A&B \cr
C&D\end{pmatrix};\sH,\sM,\sN\right\}$ with entries $A$, $B$, and $C$
given by \eqref{BOPR} then the parameter $Z(\lambda)$ is the
transfer function of the passive system
$\nu=\left\{\begin{pmatrix}D_{F^*}LD_G& F\cr
G&0\end{pmatrix};\sH,\sD_{D},\sD_{D^*}\right\}$. Moreover, we prove
that
\begin{enumerate}
\def\labelenumi{\rm (\roman{enumi})}
\item
the correspondence
\[
\tau=\left\{\begin{pmatrix} A&B \cr
C&D\end{pmatrix};\sH,\sM,\sN\right\}\longleftrightarrow
\nu=\left\{\begin{pmatrix}D_{F^*}LD_G& F\cr
G&0\end{pmatrix};\sH,\sD_{D},\sD_{D^*}\right\}
\]
preserves the properties of the system to be isometric,
co-isometric, conservative, controllable, observable, simple,
optimal, and $(*)$-optimal,
\item
 $\f_Z(\lambda)=0\iff \f_\Theta(\lambda)=0$,
$\psi_Z(\lambda)=0\iff\psi_\Theta(\lambda)=0$,
\item
the KYP inequalities for the systems $\tau$ and $\nu$  are
equivalent,
 \item the inequality
\begin{equation}
\label{MAKYP} \left\{
\begin{split}
&(I_\sH-X)P_\sH\le \left(D^2_T+T^*(I_\sH-X)P'_\sH T\right)_\sH\\
&0< X\le I_\sH
\end{split}
\right.
\end{equation}
is equivalent to the KYP inequality \eqref{kyp1}, here
$T=\begin{pmatrix} A&B \cr C&D\end{pmatrix}$ and $P_\sH$ ($P'_\sH$)
is the orthogonal projection onto $\sH$ in the Hilbert space
$\sH\oplus \sM$ ($\sH\oplus\sN$).
\end{enumerate}
 We give several equivalent forms
for the KYP inequality and establish that the minimal solution
$X_{\min}$ satisfies the algebraic Riccati equation
\[
(I_\sH-X)P_\sH=\left(D^2_T+T^*(I_\sH-X)P'_\sH T\right)_\sH.
\]
We prove that  the condition
\[
\left(D^2_{T}\right)_\sH=0\; \;(\iff \ran D_T\cap \sH=\{0\})
\] is a
necessary for the uniqueness of the solutions of \eqref{MAKYP}. In
an example it is shown
 that this condition is not sufficient. Some sufficient conditions for the
uniqueness are obtained. We show that a nondecreasing sequence
\[
X^{(0)}=0,\; X^{(n+1)}=I_\sH-\left(D^2_T +T^*(I_\sH-X^{(n)})P'_\sH
T\right)_\sH\uphar\sH\
\]
of nonnegative selfadjoint contractions strongly converges to the
minimal solution $X_{\min}$ of the KYP inequality.

\section{Shorted operators}
For every nonnegative bounded operator $S$ in the Hilbert space
$\cH$ and every subspace $\cK\subset \sH$ M.G.~Kre\u{\i}n \cite{Kr}
defined the operator $S_{\cK}$ by the formula
\[
 S_{\cK}=\max\left\{\,Z\in \bL(\cH):\,
    0\le Z\le S, \, {\ran}Z\subseteq{\cK}\,\right\}.
\]
The equivalent definition
\begin{equation}
\label{Sh1}
 \left(S_{\cK}f, f\right)=\inf\limits_{\f\in \cH\ominus{\cK}}\left\{\left(S(f + \varphi),f +
 \varphi\right)\right\},
\quad  f\in\cH.
\end{equation}
The properties of $S_{\cK}$, were studied in \cite{And, AT, AD, FW,
NA, P, Shmul}. $S_{\cK}$ is called the \textit{shorted operator}
(see \cite{And}, \cite{AT}). It is proved in \cite{Kr} that
$S_{\cK}$ takes the form
\[
 S_{\cK}=S^{1/2}P_{\Omega}S^{1/2},
\]
 where $P_{\Omega}$ is the orthogonal projection in
$\cH$ onto the subspace
\[
 \Omega=\{\,f\in \cran S:\,S^{1/2}f\in {\cK}\,\}=\cran S\ominus
S^{1/2}(\cH\ominus\cK).
\]
Hence (see \cite{Kr}),
\[
 {\ran}S_{\cK}^{1/2}={\ran}S^{1/2}P_\Omega={\ran}S^{1/2}\cap{\cK}.
\]
It follows that
\[
  S_{\cK}=0 \iff  \ran S^{1/2}\cap \cK=\{0\}.
\]
  The shortening operation possess the
following properties.
\begin{proposition} \cite{AT}.
\label{short}Let $\cK$ be a subspace in $\cH$. Then
\begin{enumerate}
\item if $S_1$ and $S_2$ are nonnegative selfadjoint operators then
$$\left(S_1+S_2\right)_\cK\ge \left(S_1\right)_\cK+\left(S_2\right)_\cK;$$
\item
$S_1\ge S_2\ge 0$ $\Rightarrow $ $\left(S_1\right)_\cK\ge
\left(S_2\right)_\cK$;
\item
if $\{S_n\}$ is a nonincreasing sequence of nonnegative bounded
selfadjoint operators and $S=s-\lim\limits_{n\to\infty}S_n$ then
$$s-\lim\limits_{n\to\infty} \left(S_n\right)_\cK=S_\cK.$$
\end{enumerate}
\end{proposition}
Let $\cK^\perp=\cH\ominus \cK$. Then a bounded selfadjoint operator
$S$ has the block-matrix form
\[
S=\begin{pmatrix}S_{11}&S_{12}\cr S^*_{12}&S_{22}
\end{pmatrix}:\begin{pmatrix}\cK\cr\cK^\perp \end{pmatrix}\to
\begin{pmatrix}\cK\cr\cK^\perp\end{pmatrix}.
\]
 According to Sylvester's criteria the
operator $S$ is nonnegative if and only if
\[
S_{22}\ge 0,\; \ran S^*_{12}\subset\ran S^{1/2}_{22},\,\; S_{11}\ge
\left(S^{-1/2}_{22}S^*_{12}\right)^*\left( S^{-1/2}_{22}S^*_{12}\right),
\]
where $S^{-1/2}_{22}$ is the Moore -- Penrose pseudoinverse operator.
Moreover, if $S\ge 0$ then the operator $S_\cK$ is given by the relation
\begin{equation}
\label{shormat}
S_\cK=\begin{pmatrix}S_{11}-\left(S^{-1/2}_{22}S^*_{12}\right)^*\left(
S^{-1/2}_{22}S^*_{12}\right)&0\cr 0&0\end{pmatrix}.
\end{equation}
If $S_{22}$ has a bounded inverse then \eqref{shormat} takes the
form
\begin{equation}
\label{shormat1} S_\cK=\begin{pmatrix}S_{11}-
S_{12}S^{-1}_{22}S^*_{12}&0\cr 0&0\end{pmatrix}.
\end{equation}
As is well known, the right hand side of \eqref{shormat1} is called
the \textit{Schur complement} of the matrix
$S=\begin{pmatrix}S_{11}&S_{12}\cr S^*_{12}&S_{22} \end{pmatrix}. $
From \eqref{shormat} it follows that
\[
S_\cK=0\iff \ran S^*_{12}\subset\ran
S^{1/2}_{22}\quad\mbox{and}\quad
S_{11}=\left(S^{-1/2}_{22}S^*_{12}\right)^*\left(
S^{-1/2}_{22}S^*_{12}\right).
\]

 The next representation of the shorted operator is new.
\begin{theorem}
\label{newshort} Let $X$ be a nonnegative selfadjoint contraction in
the Hilbert space $\cH$ and let $\cK$ be a subspace in $\cH$. Then
holds the following equality
\begin{equation}
\label{NEWSH}
(I-X)_{\cK}=P_{\cK}-P_{\cK}\left((I-X^{1/2}P_{\cK^\perp}
X^{1/2})^{-1/2}X^{1/2}P_{\cK}\right)^* (I-X^{1/2}P_{\cK^\perp}
X^{1/2})^{-1/2}X^{1/2}P_{\cK}.
\end{equation}
\end{theorem}
\begin{proof} Let us prove \eqref{NEWSH} for the case $||X||< 1$. In this case
the operator $I-X^{1/2}P_{\cK^\perp} X^{1/2}$ has bounded inverse
and
\[
P_{\cK}X^{1/2}(I-X^{1/2}P_{\cK^\perp}
X^{1/2})^{-1}X^{1/2}P_{\cK}=P_{\cK}(I-XP_{\cK^\perp})^{-1}XP_{\cK}.
\]
Let $X=\begin{pmatrix} X_{11} &X_{12} \cr X^*_{12}&
X_{22}\end{pmatrix}$ be the block-matrix representation of the
operator $X$ with respect to the decomposition
$\cH=\cK\oplus\cK^\perp$. Then
\[
(I-XP_{\cK^\perp})^{-1}=\begin{pmatrix}I&X_{12}(I-X_{22})^{-1}\cr
0&(I-X_{22})^{-1}\end{pmatrix}.
\]
Hence
 \[
P_{\cK}-P_{\cK}(I-XP_{\cK^\perp})^{-1}XP_{\cK}=
\begin{pmatrix}I-X_{11}-X_{12}(I-X_{22})^{-1}X^*_{12}&0\cr
0&0\end{pmatrix},
\]
and from \eqref{shormat1} we get \eqref{NEWSH}.

Now suppose that $||X||=1$. Then  \eqref{NEWSH} holds for the
operator $\alpha X$, where $\alpha\in(0,1)$. Let $\alpha_n=1-1/n$, $
n=1,2\ldots.$ The sequence of nonnegative selfadjoint contractions
$\{I-\alpha_n X\}$ is nonincreasing  and
$\lim_{n\to\infty}(I-\alpha_n X)=I-X$. From Proposition \ref{short}
it follows that $\lim_{n\to\infty}(I-\alpha_n X)_{\cK}=(I-X)_{\cK}$.
Since $I-X^{1/2}P_{\cK^\perp} X^{1/2}=I-X+X^{1/2}P_{\cK}X^{1/2}$, we
have
\[
\left\|((I-X^{1/2}P_{\cK^\perp}
X^{1/2})^{1/2}f\right\|^2\ge||P_{\cK}X^{1/2}f||^2,\; f\in\cH.
\]
The equality
\[
\lim\limits_{z\uparrow 0}\,\left((B-zI)^{-1}g,g\right)=\left\{
\begin{array}{ll}
    \|B^{-1/2}g\|^2, & g\in\ran B^{1/2},\\
    +\infty,           & g\in \cH\setminus\ran B^{1/2},
\end{array}\right.
\]
holds for a bounded selfadjoint nonnegative operator $B$ (see
\cite{KrO}). From R. Douglus theorem \cite{Doug}, \cite{FW} it
follows that
$$\ran
X^{1/2}P_{\cK}\subset\ran(I-X^{1/2}P_{\cK^\perp} X^{1/2})^{1/2}.$$
Hence
\[
\begin{split}
& \lim\limits_{n\to\infty}\left((I-\alpha_n
X^{1/2}P_{\cK^\perp}X^{1/2})^{-1}X^{1/2}P_{\cK}f,X^{1/2}P_{\cK}f\right)=\\
&\qquad= ||(I-X^{1/2}P_{\cK^\perp}
X^{1/2})^{-1/2}X^{1/2}P_{\cK}f||,\; f\in\cH.
\end{split}
\]
Now we arrive to \eqref{NEWSH}.
\end{proof}
\section{Parametrization of contractive block-operator matrices}
Let $\sH,\,\sK,$ $\sM$, and $\sN$ be Hilbert spaces and let $T$ be a
contraction from $\sH\oplus \sM$ into $\sK\oplus\sN.$ The following
well known result gives the parametrization of the corresponding
representation of $T$ in a block operator matrix form.
\begin{theorem} \cite{AG}, \cite{DaKaWe}, \cite {ShYa}.
\label{ParContr}The operator matrix
\[
T= \begin{pmatrix} A&B\cr C&D \end{pmatrix}:
\begin{pmatrix}\sH\cr\sM\end{pmatrix}\to
\begin{pmatrix}\sK\cr\sN\end{pmatrix}
\]
is a contraction if and only if $D\in\bL(\sM,\sN)$ is a contraction
and the entries $A$,$B$, and $C$ take the form
\[
\begin{split}
&
B=FD_D
,\; C=D_{D^*}G,\\
& A=-FD^*G+D_{F^*}LD_{G},
\end{split}
\]
where the operators $F\in\bL(\sD_D,\sK)$, $G\in\bL(\sH,\sD_{D^*})$
and $L\in\bL(\sD_{G},\sD_{F^*})$ are contractions. Moreover,
operators $F,\,G,$ and $L$ are uniquely determined.
\end{theorem}
Next we derive expressions for the shorted operators
$(D^2_{T})_\sH,$ $(D^2_{P_\sN T})_\sH,$ $(D^2_{T^*})_\sK,$ and
$(D^2_{P_\sM T})_\sK$ for a contraction $T$ given by a block matrix
form
\[
T=\begin{pmatrix}-FD^*G+D_{F^*}LD_G&FD_D\cr D_{D^*} G&D
\end{pmatrix}:\begin{pmatrix}\sH\cr\sM\end{pmatrix}\to
\begin{pmatrix}\sK\cr\sN\end{pmatrix}.
\]
By calculations from Theorem \ref{ParContr} we obtain for all
$\begin{pmatrix}f\cr h\end{pmatrix}$ and $\begin{pmatrix}g\cr
\f\end{pmatrix}$, where $f\in\sH,$  $h\in\sM$, $g\in\sK$, $\f\in
\sN$
\begin{equation}
\label{contr} \left\|D_T\begin{pmatrix}f\cr
h\end{pmatrix}\right\|^2=
||D_F\left(D_Dh-D^*Gf\right)-F^*LD_{G}f||^2+||D_{L}D_{G}f||^2,
\end{equation}
\begin{equation}
\label{*contr}
\left\|D_{T^*}\begin{pmatrix}g\cr \f\end{pmatrix}\right\|^2=
||D_{G^*}\left(D_{D^*}\f-DF^*g\right)-GL^*D_{F^*}g||^2+||D_{L^*}D_{F^*}g||^2,
\end{equation}
\begin{equation}
\label{contr1}
\left\|D_{P_\sN T}\begin{pmatrix}f\cr h\end{pmatrix}\right\|^2=
||D_Dh-D^*Gf||^2+||D_{G}f||^2,\; f\in\sH,h\in\sM,
\end{equation}
and
\begin{equation}
\label{*contr1}
\left\|D_{P_\sM T^*}\begin{pmatrix}g\cr \f\end{pmatrix}\right\|^2=
||D_{D^*}\f-DF^*g||^2+||D_{F^*}g||^2, g\in\sK,\f\in\sN.
\end{equation}
It follows from \eqref{contr}-- \eqref{*contr1} that
\[
\begin{split}
&\inf\limits_{h\in\sM}\left\{\left\|D^2_{T}\begin{pmatrix}f\cr
h\end{pmatrix}\right\|^2\right\}=||D_LD_Gf||^2,\;
\inf\limits_{h\in\sM}\left\{\left\|D^2_{P_\sN T}\begin{pmatrix}f\cr
h\end{pmatrix}\right\|^2\right\}=||D_Gf||^2,\\
&\inf\limits_{\f\in\sN}\left\{\left\|D^2_{T^*}\begin{pmatrix}g\cr
\f\end{pmatrix}\right\|^2\right\}=||D_{L^*}D_{F^*}g||^2,\;
\inf\limits_{\f\in\sN}\left\{\left\|D^2_{P_\sM
T^*}\begin{pmatrix}g\cr
\f\end{pmatrix}\right\|^2\right\}=||D_{F^*}g||^2.
\end{split}
\]
Now \eqref{Sh1} yields the following equalities for shorted
operators
\begin{equation}
\label{defshort}\left\{
\begin{split}
&\left(\left({D^2}_T\right)_\sH \begin{pmatrix}f\cr
h\end{pmatrix},\begin{pmatrix}f\cr
h\end{pmatrix}\right)=||D_LD_Gf||^2,\\
&\left(\left(D^2_{P_\sN T}\right)_\sH \begin{pmatrix}f\cr
h\end{pmatrix},\begin{pmatrix}f\cr
h\end{pmatrix}\right)=||D_Gf||^2,\; f\in\sH,\;h\in\sM,\\
&\left(\left({D^2}_{T^*}\right)_\sK \begin{pmatrix}g\cr
\f\end{pmatrix},\begin{pmatrix}g\cr
\f\end{pmatrix}\right)=||D_{L^*}D_{F^*}g||^2,\\
&\left(\left(D^2_{P_\sM T^*}\right)_\sK \begin{pmatrix}g\cr
\f\end{pmatrix},\begin{pmatrix}g\cr
\f\end{pmatrix}\right)=||D_{F^*}g||^2,\; g\in\sH,\;\f\in\sM.\\
\end{split}
\right.
\end{equation}
From \eqref{defshort} it follows that
\begin{equation}
\label{RAV}
\begin{split}
&\qquad\left({D^2}_T\right)_\sH =\left(D^2_{P_\sN T}\right)_\sH\iff
\qquad\left({D^2}_{T^*}\right)_\sK =\left(D^2_{P_\sM T^*}\right)_\sK \\
&\qquad\iff T=\begin{pmatrix}-FD^*G&FD_D\cr D_{D^*}G&D\end{pmatrix}.
\end{split}
\end{equation}
Let $D\in\bL(\sM,\sN)$ be a contraction with nonzero defect
operators and let $Q=\begin{pmatrix}S&F\cr G&0
\end{pmatrix}:\begin{pmatrix}\sH\cr\sD_{D}\end{pmatrix}\to
\begin{pmatrix}\sK\cr\sD_{D^*}\end{pmatrix}$ be a bounded operator.
 Define the transformation
 \begin{equation}
 \label{transform}
 \cM_D(Q)=
\begin{pmatrix}-FD^*G&0\cr 0&D
\end{pmatrix}+
\begin{pmatrix}I_\sK&0\cr 0&D_{D^*}
\end{pmatrix}\begin{pmatrix}S&F\cr G&0
\end{pmatrix}\begin{pmatrix}I_\sH&0\cr 0&D_D
\end{pmatrix}.
\end{equation}
Clearly, the operator $T=\cM_D(Q)$ has the following matrix form
\[
T=\begin{pmatrix}S-FD^*G&FD_D\cr D_{D^*} G&D
\end{pmatrix}:\begin{pmatrix}\sH\cr\sM\end{pmatrix}\to
\begin{pmatrix}\sK\cr\sN\end{pmatrix}.
\]
\begin{proposition}
\label{trans} Let $\sH,\sM,\sN$ be separable Hilbert spaces,
 $D\in\bL(\sM,\sN)$ be a contraction with nonzero defect operators,
$Q=\begin{pmatrix}S&F\cr G&0
\end{pmatrix}:\begin{pmatrix}\sH\cr\sD_{D}\end{pmatrix}\to
\begin{pmatrix}\sK\cr\sD_{D^*}\end{pmatrix}$ be a bounded operator,
and let
\[
T=\cM_D(Q)=\begin{pmatrix}A&B\cr C&D\end{pmatrix}:\begin{pmatrix}\sH\cr\sM
\end{pmatrix}\to \begin{pmatrix}\sH\cr\sN
\end{pmatrix}.
\]
 Then
\begin{enumerate}
\item hold the equalities
\begin{equation}
\label{min11}
\begin{split}
& \bigcap_{n=0}^\infty\ker \left(B^*A^{*n}\right)=
\bigcap_{n=0}^\infty\ker
\left(F^*S^{*n}\right),\\
&\bigcap_{n=0}^\infty\ker \left(CA^n\right)=\bigcap_{n=0}^\infty\ker
\left(GS^{n}\right),
\end{split}
\end{equation}
\item
 $T$ is a contraction if and only if $Q$ is a
contraction. $T$ is isometric (co-isometric) if and only if $Q$ is
isometric (co-isometric). Moreover, hold the equalities
\[
\begin{split}
&\left(D^2_Q\right)_\sH=\left(D^2_T\right)_\sH,\;
\left(D^2_{P_{\sD_{D^*}}Q}\right)_\sH=\left(D^2_{P_{\sN}T}\right)_\sH,\\
&\left(D^2_{Q^*}\right)_\sK=\left(D^2_{T^*}\right)_\sK,\;
\left(D^2_{P_{\sD_{D}}Q^*}\right)_\sK=\left(D^2_{P_{\sM}T^*}\right)_\sK.
\end{split}
\]
\end{enumerate}
\end{proposition}
\begin{proof}
Observe that
\[
A=-FD^*G+S,\; B=FD_D,\; C=D_{D^*}G.
\]
 Let $\Omega$ be a neighborhood of the origin and such that the resolvents
 $(I_\sH-\lambda A^*)^{-1},$ $(I_\sH-\lambda A)^{-1},$ $(I_\sH-\lambda S^*)^{-1}$, and
  $(I_\sH-\lambda S)^{-1}$ exist.
Then
\[
\begin{split}
& F^*\left(I_\sH-\lambda A^*\right)^{-1}=F^*\left(I_\sH-\lambda
S^*+\lambda G^*DF^*\right)^{-1}=\\
&=F^*\left(I_\sH+\lambda\left(I_\sH-\lambda S^*\right)^{-1}
G^*DF^*\right)^{-1}\left(I_\sH-\lambda S^*\right)^{-1}=\\
&=\left(I_\sH+\lambda F^*\left(I_\sH-\lambda
S^*\right)^{-1}G^*D\right)^{-1}F^*\left(I_\sH-\lambda
S^*\right)^{-1}.
\end{split}
\]
It follows that
\[
\bigcap_{\lambda\in\Omega}\ker\left(F^*\left(I_\sH-\lambda
A^*\right)^{-1}\right)=\bigcap_{\lambda\in\Omega}\ker\left(F^*\left(I_\sH-\lambda
S^*\right)^{-1}\right).
\]
Similarly
\[
\bigcap_{\lambda\in\Omega}\ker\left(G\left(I_\sH-\lambda
A\right)^{-1}\right)=\bigcap_{\lambda\in\Omega}\ker\left(G\left(I_\sH-\lambda
S\right)^{-1}\right).
\]
Since $B^*=D_DF^*$ and $C=D_{D^*}G$, we get \eqref{min11}.

The statement (2) is a consequence of Theorem \ref{ParContr} and
formulas \eqref{contr}--\eqref{defshort}.
\end{proof}
\begin{proposition}
\label{NEW} Let $D\in\bL(\sM,\sN)$ be a contraction with nonzero
defect operators, $Q=\begin{pmatrix}S&F\cr G&0
\end{pmatrix}:\begin{pmatrix}\sH\cr\sD_{D}\end{pmatrix}\to
\begin{pmatrix}\sH\cr\sD_{D^*}\end{pmatrix}$ be a contraction and
let $T=\cM_D(Q)$. Then for every nonnegative selfadjoint contraction
$X$ in $\sH$ hold the following equalities
\begin{equation}
\label{shortmap} \left(D^2_{T}+T^*(I_\sH-X)P'_\sH
T\right)_\sH=\left(D^2_{Q}+Q^*(I_\sH-X)P'_\sH Q\right)_\sH,
\end{equation}
\begin{equation}
\label{EXSHORT}
\begin{split}
&\left(D^2_Q+Q^*(I_\sH-X)P'_\sH Q\right)_\sH=\\
&\qquad\qquad=\left(D^2_G-D_GL^*D_{F^*}X^{1/2}(I_\sH-X^{1/2}FF^*X^{1/2})^{-1}X^{1/2}D_{F^*}LD_G\right)P_\sH,
\end{split}
\end{equation}
where $P_\sH$ ($P'_\sH$) is the orthogonal projection in
$\cH=\sH\oplus\sM$ ($\cH'=\sH\oplus\sN)$ onto $\sH$, and
\[
\begin{split}
&D_{F^*}X^{1/2}(I_\sH-X^{1/2}FF^*X^{1/2})^{-1}X^{1/2}D_{F^*}:=\\
&=\left((I_\sH-X^{1/2}FF^*X^{1/2})^{-1/2}X^{1/2}D_{F^*}\right)^*
\left((I_\sH-X^{1/2}FF^*X^{1/2})^{-1/2}X^{1/2}D_{F^*}\right).
\end{split}
\]
\end{proposition}
\begin{proof}
Define the contraction
\begin{equation}
\label{wtQ} \wt Q=\begin{pmatrix}X^{1/2}S&X^{1/2}F\cr G&0
\end{pmatrix}:\begin{pmatrix}\sH\cr\sD_{D}\end{pmatrix}\to
\begin{pmatrix}\sH\cr\sD_{D^*}\end{pmatrix}.
\end{equation}
Then
\[
D^2_{\wt Q}=D^2_{Q}+Q^*(I_\sH-X)P'_\sH Q.
\]
By Proposition \ref{trans} the operator $\wt T=\cM_D(\wt Q)$ is a
contraction as well. Clearly,
\[
D^2_{\wt T}=D^2_{T}+T^*(I_\sH-X)P'_\sH T
\]
Applying once again Proposition \ref{trans} we arrive to
\eqref{shortmap}.

 Since $Q$ is a contraction, by Theorem \ref{ParContr} the operator $S$ takes the form
$S=D_{F^*}LD_G$, where $L\in\bL(\sD_G,\sD_{F^*})$ is a contraction.
Because $\wt Q$ given by \eqref{wtQ} is a contraction, we get
$X^{1/2}S=D_{\wt F^*}\wt LD_G$, where $\wt F=X^{1/2}F$ and $\wt
L\in\bL(\sD_G,\sD_{\wt F^*})$ is a contraction. Since
\[
\begin{split}
&D^2_{\wt F^*}=I_\sH-\wt F\wt F^*=I_\sH-X^{1/2}FF^*X^{1/2}=I_\sH -X+X-X^{1/2}FF^*X^{1/2}=\\
&\qquad\qquad=I_\sH-X+X^{1/2}D^2_{F^*}X^{1/2},
\end{split}
\]
we have $||D_{\wt F^*}f||^2\ge ||D_{F^*}X^{1/2}f||^2$ for all
$f\in\sH.$ Using R.~Douglas theorem \cite{Doug}, \cite{FW} we
conclude that $\ran D_{\wt F^*}\supset \ran(X^{1/2}D_{F^*})$. Let
$D^{-1}_{\wt F^*}=(I_\sH-X^{1/2}FF^*X^{1/2})^{-1/2}$ be the
Moore-Penrose inverse for $D_{\wt F^*}$. Then we obtain the equality
\[
\wt L=D^{-1}_{\wt F^*}(X^{1/2}D_{F^*})L.
\]
The first equality in \eqref{defshort} yields $\left(D^2_{\wt
Q}\right)_\sH=D_GD^2_{\wt L}D_GP_\sH$. Since
\[
\left(D^2_{\wt Q}\right)_\sH=\left(D^2_Q+Q^*(I_\sH-X)P'_\sH
Q\right)_\sH,
\]
we get \eqref{EXSHORT}.
\end{proof}

\section {The M\"obius representations}
Let $T:H_1\to H_2$ be a contraction and let $\cV_{T^*}$ be the set
of all contractions $Z\in\bL(\sD_T,\sD_{T^*})$ such that
$-1\in\rho(T^*Z).$ In \cite{Shmul1} the fractional-linear
transformations
\begin{equation}
\label{unflt} \cV_{T^*}\ni Z\to Q=T+D_{T^*}Z(I_{\sD_T}+T^*Z)^{-1}D_T
\end{equation}
was studied and the following result has been established.
\begin{theorem}
\cite{Shmul1} \label{SHMUL}. Let  $T\in\bL(H_1,H_2)$ be a
contraction and let $Z\in\cV_{T^*}$. Then
$Q=T+D_{T^*}Z(I_{\sD_T}+T^*Z)^{-1}D_T$ is a contraction,
\[
||D_{Q}f||^2=||D_Z(I_{\sD_T}+T^*Z)^{-1}D_Tf||,\; f\in H_1,
\]
$\ran D_Q\subseteq\ran D_T$, and $\ran D_Q=\ran D_T$ if and only if
$||Z||< 1.$ Moreover, if $Q\in\bL(H_1,H_2)$ is a contraction and
$Q=T+D_{T^*}XD_T,$ where $X\in\bL(\sD_T,\sD_{T^*})$ then
\[2\,\RE\left((I_{\sD_T}-T^*X)f,f\right)\ge ||f||^2\]
for all $f\in\sD_{T^*}$, the operator $Z=X(I_{\sD_T}-T^*X)^{-1}$
belongs to $\cV_{T^*},$ and
$$Q=T+D_{T^*}Z(I_{\sD_T}+T^*Z)^{-1}D_T.$$
\end{theorem}
The transformation \eqref{unflt} is called in \cite{Shmul1} the
unitary linear-fractional transformation. It is not dificult to see
that if $||T||<1$ then the closed unit operator ball in $\bL(H_1,
H_2)$ belongs to the set $\cV_{T^*}$ and, moreover
\[
\begin{split}
&T+D_{T^*}Z(I_{H_1}+T^*Z)^{-1}D_T=D^{-1}_{T^*}(Z+T)(I_{H_1}+T^*Z)^{-1}D_{T}=\\
&\qquad\qquad\qquad=D_{T^*}(I_{H_2}+ZT^*)^{-1}(Z+T)D^{-1}_{T}
\end{split}
\]
for all $Z\in\bL(H_1,H_2),$ $||Z||\le 1.$ Thus, the transformation
\eqref{unflt} is an operator analog of a well known M\"obius
transformation of the complex plane
\[
z\to \frac{z+t}{1+\bar t z},\;|t|\le 1.
\]
The next theorem is a version of the more general result established
by Yu.L.~Shmul'yan in \cite{Shmul2}.
\begin{theorem}
\label{SHMUL1} Let $\sM$ and $\sN$ be Hilbert spaces and let the
function $\Theta(\lambda)$ be from the Schur class ${\bf
S}(\sM,\sN).$ Then
\begin{enumerate}
\item the linear manifolds
$\ran D_{\Theta(\lambda)}$ and $\ran D_{\Theta^*(\lambda)}$ do not depend on $\lambda\in\dD,$
\item for arbitrary $\lambda_1,$ $\lambda_2$, $\lambda_3$ in $\dD$ the function $\Theta(\lambda)$ admits the representation
\[
\Theta(\lambda)=\Theta(\lambda_1)+D_{\Theta^*(\lambda_2)}\Psi(\lambda)D_{\Theta(\lambda_3)},
\]
where $\Psi(\lambda)$ is
$\bL\left(\sD_{\Theta(\lambda_3)},\sD_{\Theta^*(\lambda_2)}\right)$-valued
function holomorphic in $\dD$.
\end{enumerate}
\end{theorem}
Now using Theorems \ref{SHMUL} and \ref{SHMUL1} we obtain the
following result (cf. \cite{BC}).
\begin{theorem}
\label{MO} Let $\sM$ and $\sN$ be Hilbert spaces and let the
function $\Theta(\lambda)$ be from the Schur class ${\bf
S}(\sM,\sN).$ Then there exists a unique function $Z(\lambda)$ from
the Schur class ${\bf S}(\sD_{\Theta(0)},\sD_{\Theta^*(0)})$ such
that
\[
\Theta(\lambda)=\Theta(0)+
D_{\Theta^*(0)}Z(\lambda)(I_{\sD_{\Theta(0)}}+\Theta^*(0)Z(\lambda))^{-1}D_{\Theta(0)},\;\lambda\in\dD.
\]
\end{theorem}
In what follows we will say that the right hand side of the above
equality is \textit{the M\"obius representation} and the function
$Z(\lambda)$ is the \textit{the M\"obius parameter}  of
$\Theta(\lambda)$. Clearly, $Z(0)=0$ and by Schwartz's lemma we
obtain that
\[
||Z(\lambda)||\le|\lambda|,\;\lambda\in\dD.
\]
The next result provides connections between the realizations of
$\Theta(\lambda)$ and  $Z(\lambda)$.
\begin{theorem}
\label{TT1} \begin{enumerate}
\item Let $\tau=\left\{\begin{pmatrix}A&B\cr C&D
\end{pmatrix};\sH,\sM,\sN\right\}$ be a passive system and let
\[
T=\begin{pmatrix}A&B\cr C&D
\end{pmatrix}
=\begin{pmatrix}-FD^*G+D_{F^*}LD_{G}&FD_D\cr D_{D^*}G&D
\end{pmatrix}:\begin{pmatrix}\sH\cr\sM\end{pmatrix}\to
\begin{pmatrix}\sH\cr\sN\end{pmatrix}.
\]
Let $\Theta(\lambda)$ be the transfer function of $\tau$.
 Then
 \begin{enumerate}
\def\labelenumi{\rm (\roman{enumi})}
 \item the M\"obius parameter $Z(\lambda)$ of the function $\Theta(\lambda)$ is the
transfer function of the passive system
\[
\nu=\left\{\begin{pmatrix}D_{F^*}LD_G& F\cr
G&0\end{pmatrix};\sH,\sD_{D},\sD_{D^*}\right\};
\]
\item
the system $\tau$ isometric (co-isometric) $\Rightarrow$ the system
$\nu$ isometric (co-isometric);
\item  the equalities $\sH^c_\nu =\sH^c_\tau$, $\sH^o_\nu =\sH^o_\tau$ hold and hence the system $\tau$ is
controllable (observable) $\Rightarrow$ the system $\nu$ is
controllable (observable),
 the system $\tau$ is simple (minimal)
$\Rightarrow$ the system $\nu$ is simple (minimal).
\end{enumerate}
\item
 Let a nonconstant function $\Theta(\lambda)$ belongs to the Schur class
${\bf S}(\sM,\sN)$ and let $Z(\lambda)$ be the M\"obius parameter of
the function $\Theta(\lambda)$. Suppose that the transfer function
of the linear system
\[ \nu'=\left\{\begin{pmatrix}S&F\cr
G&0\end{pmatrix};\sH,\sD_{\Theta(0)},\sD_{\Theta^*(0)}\right\}
\]
coincides with $Z(\lambda)$ in a neighborhood of the origin. Then
the transfer function of the linear system
\[
\tau'=\left\{\begin{pmatrix}-F\Theta^*(0)G+S&FD_{\Theta(0)}\cr
D_{\Theta^*(0)}G&\Theta(0)\end{pmatrix};\sH,\sM,\sN\right\}
\]
coincides with
 $\Theta(\lambda)$ in a neighborhood of the origin. Moreover
 \begin{enumerate}
 \item
 the equalities $\sH^c_{\tau'} =\sH^c_{\nu'}$, $\sH^o_{\tau'} =\sH^o_{\nu'}$ hold, and hence the system
$\nu'$ is controllable (observable) $\Rightarrow$ the system $\tau'$
is controllable (observable), the system $\nu'$ is simple (minimal)
$\Rightarrow$ the system $\tau'$ is simple (minimal),
\item the system $\nu'$ is passive $\Rightarrow$ the system $\tau'$
is passive,
\item
the system $\nu'$ isometric (co-isometric) $\Rightarrow$ the system
$\tau'$ isometric (co-isometric).
\end{enumerate}
\end{enumerate}
\end{theorem}
\begin{proof}Suppose that $D\in\bL(\sM,\sN)$ is a contraction with nonzero defects.
Given the operator matrices $Q=\begin{pmatrix}S&F\cr
G&0\end{pmatrix}$: $\begin{pmatrix}\sH\cr\sD_{D}\end{pmatrix}$ $\to$
$\begin{pmatrix}\sH\cr\sD_{D^*}\end{pmatrix}$. Let
 \[
T=\cM_D(Q)=\begin{pmatrix}S-FD^*G&FD_D\cr D_{D^*}G&D\end{pmatrix}:
\begin{pmatrix}\sH\cr\sM\end{pmatrix}\to
\begin{pmatrix}\sH\cr\sN\end{pmatrix}
\]
and let $\Omega$ be a sufficiently small neighborhood of the origin.
Consider the linear systems
\[
\left\{\begin{pmatrix}S&F\cr
G&0\end{pmatrix};\sH,\sD_D,\sD_{D^*}\right\}\;\mbox{and}\;\left\{\begin{pmatrix}-FD^*G+S&FD_D\cr
D_{D^*}G&D\end{pmatrix};\sH,\sM,\sN\right\}.
\]
Define the transfer functions
\[
\begin{split}
&Z(\lambda)=\lambda G(I_\sH-\lambda S)^{-1} \quad \mbox{and}\\
&\Theta(\lambda)=D+\lambda
D_{D^*}G\left(I_\sH-\lambda(S-FD^*G)\right)^{-1}FD_D
\end{split}
\]
 Since $\Theta(0)=D$, we have for $\lambda\in\Omega$
\[
\begin{split}
&Z(\lambda)(I_{\sD_D}+\Theta^*(0)Z(\lambda))^{-1}=\\
&\quad=\lambda G(I_\sH-\lambda S)^{-1}F
\left(I_{\sD_D}+\lambda D^*G(I_\sH-\lambda S)^{-1}F\right)^{-1}=\\
&=\lambda G(I_\sH-\lambda S)^{-1}
\left(I_\sH+\lambda FD^*G(I_\sH-\lambda S)^{-1}\right)^{-1}F=\\
&=\lambda G\left(I_\sH-\lambda S+\lambda FD^*G\right)^{-1}F.
\end{split}
\]
Hence
\[
\Theta(\lambda)=\Theta(0)+D_{\Theta^*(0)}Z(\lambda)(I_{\sD_{\Theta(0)}}+\Theta^*(0)Z(\lambda))^{-1}D_{\Theta(0)},\;\lambda\in\Omega.
\]
According to Theorem \ref{ParContr} the operator $Q$ is a
contraction if and only if $F,G$ are contractions and
$S=D_{F^*}LD_G$, where $L\in\bL(\sD_G,\sD_{F^*})$ is a contractions.
Now from Proposition \ref{trans} we get that all statements of
Theorem \ref{TT1} hold true.
\end{proof}
\begin{corollary}
\label{zero}
The equivalences
\[
\begin{array}{l}
\f_\Theta(\lambda)=0 \iff \f_Z(\lambda)=0,\\
\psi_\Theta(\lambda)=0\iff \psi_Z(\lambda)=0
\end{array}
\]
hold.
\end{corollary}
\begin{proof}
 Let $\f_\Theta(\lambda)=0$ ($\psi_\Theta(\lambda)=0$) and let
$\tau=\left\{\begin{pmatrix}A&B\cr C&D
\end{pmatrix};\sH,\sM,\sN\right\}$ be a simple conservative system with transfer function
$\Theta(\lambda)$. By Theorem \ref{DBR} the system $\tau$ is
observable (controllable). By Theorem  \ref{TT1} the corresponding
system $\nu$ with transfer function $Z(\lambda)$ is conservative and
observable (controllable). Theorem \ref{DBR} yields that
$\f_Z(\lambda)=0$ ($\psi_Z(\lambda)=0$).

Conversely. Let $\f_Z(\lambda)=0$ ($\psi_Z(\lambda)=0$ and let
$\nu'$ be a simple conservative system with transfer function
$Z(\lambda)$. Again by Theorem \ref{DBR} the system $\nu'$ is
observable (controllable). By Theorem  \ref{TT1} the corresponding
system $\tau'$ with transfer function $\Theta(\lambda)$ is
conservative and observable (controllable) as well and Theorem
\ref{DBR} yields that $\f_\Theta(\lambda)=0$
($\psi_\Theta(\lambda)=0$).
 \end{proof}
 The next statement follows from Theorem \ref{INNERSYS}.
\begin{corollary}
\label{isomreal} Let $\Theta(\lambda)\in{\bf S}(\sM,\sN)$.
\begin{enumerate}
\item Suppose that $\f_\Theta(\lambda)=0$ ($\psi_\Theta(\lambda)=0$). Then
every passive controllable (observable) realization of the M\"obius
parameter $Z(\lambda)$ of $\Theta$ is isometric (co-isometric) and
minimal system;
\item  Suppose that  $\varphi_\Theta(\lambda)=0$ and
$\psi_\Theta(\lambda)=0.$ Then every  simple and passive realization
of the M\"obius parameter $Z(\lambda)$ of $\Theta(\lambda)$ is
conservative and minimal.
\end{enumerate}
\end{corollary}
\begin{proposition}
\label{LIN} Let $\Theta(\lambda)$ be a function from the Schur class
${\bf S}(\sM,\sN)$. Suppose that the M\"obius parameter $Z(\lambda)$
of  $\Theta(\lambda)$ is a linear function of the form
$Z(\lambda)=\lambda K,$ $||K||\le 1$. Then there exists a passive
realization $\tau=\left\{\begin{pmatrix}A&B\cr
C&D\end{pmatrix};\sH,\sM,\sN\right\}$ such that
\begin{equation}
\label{linej} \left (D^2_{P_\sN T}\right)_\sH=\left
(D^2_{T}\right)_\sH,
\end{equation}
where $T=\begin{pmatrix}A&B\cr C&D \end{pmatrix}$.

 Conversely, if a passive system
$\tau=\left\{\begin{pmatrix}A&B\cr
C&D\end{pmatrix};\sH,\sM,\sN\right\}$ possess the property
\eqref{linej} then the M\"obius parameter $Z(\lambda)$ of the
transfer function $\Theta(\lambda)$ of $\tau$ is a linear function
of the form $\lambda K$.
\end{proposition}
\begin{proof}
Let $Z(\lambda)=\lambda K,$ $\lambda\in\dD,$ where
$K\in\bL(\sD_{\Theta(0)},\sD_{\Theta^*(0)})$ is a contraction. Then
$Z(\lambda)$ can be realized as the transfer function of a passive
system of the form
\[
\nu=\left\{\begin{pmatrix}0&F\cr G&0\end{pmatrix};
\sH,\sD_{\Theta(0)},\sD_{\Theta^*(0)}\right\}.
\]
Actually, take
$$\sH=\cran K, \;F=K,\; G=j,$$
where $j$ is the embedding of $\cran K$ into $\sD_{\Theta^*(0)}.$ It
follows that $GF=K$.
 By Theorem
\ref{TT1} the system $ \tau=\left\{\begin{pmatrix}-F\Theta^*(0)G&F
D_{\Theta(0)}\cr
D_{\Theta^*(0)}G&\Theta(0)\end{pmatrix};\sH,\sM,\sN\right\}$  is a
passive realization of the function $\Theta(\lambda)$. From
\eqref{RAV} it follows that $\left (D^2_{P_\sN T}\right)_\sH=\left
(D^2_{T}\right)_\sH$ for $T=\begin{pmatrix}-F\Theta^*(0)G&F
D_{\Theta(0)}\cr D_{\Theta^*(0)}G&\Theta(0)\end{pmatrix}.$

Assume now $ \left (D^2_{P_\sN T}\right)_\sH=\left
(D^2_{T}\right)_\sH,$ where $T=\begin{pmatrix}A&B\cr
C&D\end{pmatrix}:\begin{pmatrix}\sH\cr\sM\end{pmatrix}\to
\begin{pmatrix}\sH\cr\sN\end{pmatrix}$ is a contraction and  let $\Theta(\lambda)$ be the transfer function
of the system $\tau=\left\{\begin{pmatrix}A&B\cr
C&D\end{pmatrix};\sH,\sM,\sN\right\}$.  Then the entries $A,$ $B,$
and $C$ takes the form \eqref{BOPR}, and $D=\Theta(0).$ According to
\eqref{RAV} we have $D_{F^*}LD_G=0,$ i.e.
 $T=\begin{pmatrix}-F\Theta^*(0)G&FD_{\Theta(0)}\cr
 D_{\Theta^*(0)}G&\Theta(0)\end{pmatrix}.$ By Theorem \ref{TT1}
the M\"obius parameter $Z(\lambda)$ of $\Theta(\lambda)$ takes the
form $Z(\lambda)=\lambda GF.$
\end{proof}
\begin{example}
\label{EXX} Let $A$ be completely non-unitary contractions in the
Hilbert space $\sH$ and let $\Phi(\lambda)$ be the Sz.Nagy--Foias
characteristic function of $A$* \cite{SF}:
\[
\Phi(\lambda)=\left(-A^*+\lambda D_A(I_\sH-\lambda
A)^{-1}D_{A^*}\right)\uphar\sD_{A^*}:\sD_{A^*}\to\sD_A,\;
|\lambda|<1.
\]
The system
\[
\tau=\left\{\begin{pmatrix}A&D_{A^*}\cr D_{A}&-A^*\end{pmatrix};
\sH,\sD_{A^*},\sD_A\right\}
\]
is conservative and simple. Let
\[
\Phi(\lambda)=\Phi(0)+D_{\Phi^*(0)}Z(\lambda)(I_{\sD_{\Phi(0)}}+\Phi^*(0)Z(\lambda))^{-1}D_{\Phi(0)},\;\lambda\in\dD
\]
be the M\"obius representation of the function $\Phi(\lambda)$.
Since $F$ and $G^*$ are imbedding  of the subspaces $\sD_{A^*}$  and
$\sD_{A}$ into $\sH$, we get that
\[
D_{F^*}=P_{\ker D_{A^*}},\;D_{G}=P_{\ker D_{A}}
\]
and $L=A\uphar\ker\sD_A$ is isometric operator. Let
\[
\nu=\left\{\begin{pmatrix}AP_{\ker D_A}& I\cr
P_{\sD_A}&0\end{pmatrix}:\sH,\sD_{A^*},\sD_A \right\}.
\]
By Theorem \ref{TT1}
\[
Z(\lambda)=\lambda P_{\sD_A}\left(I_\sH-\lambda AP_{\ker
D_{A}}\right)^{-1}\uphar\sD_{A^*},\; |\lambda|<1
\]
and this function is transfer function of $\nu$. Note that this
function is the Sz.-Nagy--Foias characteristic function of the
partial isometry $A^*P_{\ker\sD_{A^*}}$.
\end{example}

\section{The Kalman--Yakubovich--Popov inequality and Riccati equation}
Let $\sH,\,\sM$ and $\sN$ be Hilbert spaces and let $T$ be a bounded
linear operator from the Hilbert space $\cH=\sH\oplus\sM$ into the
Hilbert space $\cH'=\sH\oplus\sN$ given by the block matrix
\[
T=\begin{pmatrix} A&B
\cr C&D\end{pmatrix} : \begin{pmatrix} \sH \\ \sM \end{pmatrix} \to
\begin{pmatrix} \sH \\ \sN \end{pmatrix}.
\]
Suppose that $X$ is a positive selfadjoint operator in the Hilbert space $\sH$
such that
\[
A\,\dom X^{1/2}\subset \dom X^{1/2},\;\ran B\subset\dom X^{1/2}.
\]
As was mentioned in Introduction the inequality
\eqref{kyp3}
\[
\begin{split}
&\left\|\begin{pmatrix}{X}^{1/2}&0\cr 0&I_\sM\end{pmatrix}\begin{pmatrix}x\cr u\end{pmatrix}
\right\|^2-\left\|\begin{pmatrix}{X}^{1/2}&0\cr 0&I_\sN\end{pmatrix}\begin{pmatrix} A&B
\cr C&D\end{pmatrix}\begin{pmatrix}x\cr u\end{pmatrix}\right\|^2\ge 0\\
& \qquad\mbox{for all}\quad\; x\in\dom X^{1/2},\, u\in\sM
\end{split}
\]
is called the generalized  KYP inequality with respect to $X$
\cite{ArKaaP}, \cite{ArKaaP3}. For a bounded solution $X$ the KYP
inequality \eqref{kyp3} takes the form \eqref{kyp1}.

Put
\[
{\bf \wh X}:=\begin{pmatrix}{X}&0\cr 0&I_\sM\end{pmatrix},\;
{\bf \wt X}:=\begin{pmatrix}{X}&0\cr 0&I_\sN\end{pmatrix}.
\]
Operators ${\bf \wh X}$ and ${\bf \wt X}$ are positive selfadjoint
operators in Hilbert spaces $\cH$ and $\cH'$ respectively, $\dom
{\bf \wh X}=\dom X\oplus\sM$, $\dom {\bf \wt X}=\dom X\oplus\sN$.
Let the operator $X$ satisfies the KYP inequality. Let us define the
operator $T_1:$
\begin{equation}
\label{OPT1}
\begin{split}
&\dom T_1:=\ran {\bf\wh X}^{1/2}=\ran X^{1/2}\oplus\sM,\\
&T_1:={\bf \wt X}^{1/2}T{\bf \wh X}^{-1/2}=\begin{pmatrix}{X}^{1/2}&0\cr 0&I_\sN\end{pmatrix}
T\begin{pmatrix}{X}^{-1/2}&0\cr 0&I_\sM\end{pmatrix}=\\
&\qquad=\begin{pmatrix} X^{1/2}AX^{-1/2}&X^{1/2}B
\cr CX^{-1/2}&D\end{pmatrix}.
\end{split}
\end{equation}
Clearly, the following statements are equivalent:
\begin{enumerate}
\item
 $X$ is a solution of the KYP inequality \eqref{kyp3};
 \item
the operator $T_1$ is densely defined contraction, i.e.
\[
\left\|\begin{pmatrix}{X}^{1/2}x\cr u\end{pmatrix}
\right\|^2-\left\|T_1\begin{pmatrix}{X}^{1/2}x\cr
u\end{pmatrix}\right\|^2\ge 0,\; x\in\dom X^{1/2},\, u\in\sM;
\]
 \item
the operator $T$ is a contraction acting from a pre-Hilbert space
$\dom {\bf\wh X}$ into a pre-Hilbert space $\dom {\bf\wt X}$
equipped by the inner products
\[
\begin{split}
&\left(\begin{pmatrix}x_1\cr u_1\end{pmatrix},\begin{pmatrix}x_2\cr
u_2\end{pmatrix}\right)=(X^{1/2}x_1,X^{1/2}x_2)_\sH+(u_1,u_2)_\sM,\\
&\left(\begin{pmatrix}x_1\cr v_1\end{pmatrix},\begin{pmatrix}x_2\cr
v_2\end{pmatrix}\right)=(X^{1/2}x_1,X^{1/2}x_2)_\sH+(v_1,v_2)_\sN,\\
&x_1,x_2\in\dom X^{1/2},\; u_1,u_2\in\sM,\; v_1,v_2\in\sN;
\end{split}
\]
\item
$Z=X^{-1}$ is the solution of the generalized KYP inequality for the
adjoint operator
\begin{equation}
\label{kyp3*}
\begin{split}
&\left\|\begin{pmatrix}{Z}^{1/2}&0\cr
0&I_\sN\end{pmatrix}\begin{pmatrix}x\cr v\end{pmatrix}
\right\|^2-\left\|\begin{pmatrix}{Z}^{1/2}&0\cr
0&I_\sM\end{pmatrix}\begin{pmatrix} A^*&C^*
\cr B^*&D^*\end{pmatrix}\begin{pmatrix}x\cr v\end{pmatrix}\right\|^2\ge 0\\
& \qquad\mbox{for all}\quad\; x\in\dom Z^{1/2},\, v\in\sN.
\end{split}
\end{equation}
\end{enumerate}

Let the positive selfadjoint operator $X$ in $\sH$ satisfies the KYP
inequality. If
\begin{equation}
\label{RIC1}
\begin{split}
&\inf\limits_{u\in\sM}\left\{\left\|{\bf \wh X}^{1/2}\begin{pmatrix}x\cr u\end{pmatrix}\right\|^2-
\left\|{\bf \wt X}^{1/2}T\begin{pmatrix}x\cr u\end{pmatrix}\right\|^2\right\}=0\\
&\quad\mbox{for all} \quad x\in\dom X^{1/2}
\end{split}
\end{equation}
we will say that the operator $X$ satisfies \textit{the Riccati equation}.
\begin{proposition}
\label{SHORTSYST} If the positive selfadjoint operator $X$ satisfies the Riccati equation
then the continuation of the operator $T_1$ defined by \eqref{OPT1}meets the condition
\[
\left(D^2_{T_1}\right)_\sH=0.
\]
\end{proposition}
\begin{proof}
By \eqref{OPT1} we get for all $\vec f\in\dom {\bf\wh X}^{1/2}$
\[
\begin{split}
&\left\|D_{T_1}{\bf\wh X}^{1/2}\vec f\right\|^2=||{\bf\wh X}^{1/2}\vec f||^2-||T_1{\bf\wh X}^{1/2}\vec f||^2=\\
&\quad=||{\bf\wh X}^{1/2}\vec f||^2-||{\bf\wt X}^{1/2}T\vec f||^2.
\end{split}
\]
Since \eqref{RIC1} holds, we get
\[
\inf\limits_{u\in\sM}\left\{\left\|D_{T_1}\begin{pmatrix} X^{1/2}x\cr u\end{pmatrix}\right\|^2\right\}=0
\]
for all $x\in\dom X^{1/2}$ and all $u\in\sM$. Because $\ran X^{1/2}$
is dense in $\sH$, we obtain $\left(D^2_{T_1}\right)_\sH=0.$
\end{proof}
\begin{proposition}
\label{EQ}
  Let $D\in\bL(\sM,\sN)$ be a contraction with nonzero defect operators.
Let
\[
 Q=\begin{pmatrix}S&F\cr G&0\end{pmatrix}:\begin{pmatrix}
 \sH\cr\sD_D\end{pmatrix}\to\begin{pmatrix} \sH\cr\sD_{D^*}\end{pmatrix}
 \]
 and let
 \[
 T=\cM_D(Q): \begin{pmatrix} \sH \cr \sM \end{pmatrix} \to
\begin{pmatrix} \sH \cr \sN \end{pmatrix}.
 \]
Then
\begin{enumerate}
\item
the KYP inequality\eqref{kyp3} and the KYP inequality
\[
\begin{split}
&\left\|\begin{pmatrix}{X}^{1/2}&0\cr
0&I_{\sD_D}\end{pmatrix}\begin{pmatrix}x\cr u\end{pmatrix}
\right\|^2-\left\|\begin{pmatrix}{X}^{1/2}&0\cr
0&I_{\sD_{D^*}}\end{pmatrix}\begin{pmatrix} S&G
\cr F&0\end{pmatrix}\begin{pmatrix}x\cr u\end{pmatrix}\right\|^2\ge 0\\
& \qquad\mbox{for all}\quad\; x\in\dom X^{1/2},\, u\in\sD_D
\end{split}
\]
are equivalent,
\item the Riccati equation \eqref{RIC1} and the Riccati equation
\[
\begin{split}
&\inf\limits_{u\in\sD_{D}}\left\{\left\|\begin{pmatrix}X^{1/2}x\cr
u\end{pmatrix}\right\|^2-
\left\|\begin{pmatrix} X^{1/2}&0\cr 0& I_{\sD_{D^*}}\end{pmatrix}Q\begin{pmatrix}x\cr u\end{pmatrix}\right\|^2\right\}=0\\
&\quad\mbox{for all} \quad x\in\dom X^{1/2}
\end{split}
\]
are equivalent.
\end{enumerate}
\end{proposition}
\begin{proof}
From \eqref{transform} we have the relation
\[
\cM_D\left(\begin{pmatrix}{X^{1/2}}&0\cr
0&I_{\sD_{D^*}}\end{pmatrix}Q
\begin{pmatrix}{X^{-1/2}}&0\cr 0&I_{\sD_{D}}\end{pmatrix}\right)=\begin{pmatrix}{X^{1/2}}&0\cr
0&I_\sN\end{pmatrix}\cM_D(Q)\begin{pmatrix}{X^{-1/2}}&0\cr
0&I_\sM\end{pmatrix}.
\]
Now the result follows from Propositions \ref{trans} and
\ref{SHORTSYST}.
\end{proof}

  \section{Equivalent forms of the KYP inequality and Riccati equation for a passive system}
\begin{theorem}
\label{EQIN} Let
\[ T=\begin{pmatrix} A&B \cr
C&D\end{pmatrix}=\begin{pmatrix}-FD^*G+D_{F^*}LD_G&FD_D\cr D_{D^*}
G&D\end{pmatrix}:
\begin{pmatrix}\sH\cr\sM\end{pmatrix}\to
\begin{pmatrix}\sH\cr\sN\end{pmatrix}
\]
 be a contraction and let
\[
Q=\begin{pmatrix}D_{F^*}LD_G&F\cr
G&0\end{pmatrix}:\begin{pmatrix}\sH\cr\sD_D\end{pmatrix}\to
\begin{pmatrix}\sH\cr \sD_{D^*}\end{pmatrix}.
\]
 Then the following inequalities are equivalent
\begin{equation}
\label{CKYP}
\left\{
\begin{split}
&\begin{pmatrix}{X}&0\cr 0&I_\sM\end{pmatrix}-T^*\begin{pmatrix}{X}&0\cr 0&I_\sN\end{pmatrix}
T\ge 0\\
&0< X\le I_\sH
\end{split}
\right. \;,
\end{equation}
\[
 \left\{
\begin{split}
&\begin{pmatrix} X-A^* XA-C^*C & -A^*X B- C^*D\cr -B^*X A-D^* C &
I_\sM- B^*X B-D^*D
\end{pmatrix}\ge 0\\
 &0< X\le I_\sH
\end{split}
\right.\;,
\]
\begin{equation}
\label{SkypX}
\left\{
\begin{split}
&(I_\sH-X)P_\sH\le \left(D^2_T+T^*(I_\sH-X)P'_\sH T\right)_\sH\\
&0< X\le I_\sH
\end{split}
\right.\;,
\end{equation}
\begin{equation}
\label{CKYP1}
\left\{
\begin{split}
&\begin{pmatrix} X-G^*G&D_G L^*D_{F^*}X^{1/2}\cr X^{1/2}D_{F^*}LD_G& I_\sH-X^{1/2}FF^*X^{1/2} \end{pmatrix}\ge 0\\
&0< X\le I_\sH
\end{split}
\right.\,,
\end{equation}
\begin{equation}
\label{SHORTX}
\left\{
\begin{split}
& X\ge G^*G+D_GL^*D_{F^*}X^{1/2}(I_\sH-X^{1/2}FF^*X^{1/2})^{-1}X^{1/2}D_{F^*}LD_G\\
&0< X\le I_\sH
\end{split}
\right.,
\end{equation}
\begin{equation}
\label{CKYPQ}
\left\{
\begin{split}
&\begin{pmatrix}{X}&0\cr 0&I_{\sD_{D}}\end{pmatrix}-Q^*\begin{pmatrix}{X}&0\cr 0&I_{\sD_{D^*}}\end{pmatrix}
Q\ge 0\\
&0< X\le I_\sH
\end{split}
\right.,
\end{equation}
\begin{equation}
\label{SkypQXX}
 \left\{
\begin{split}
&\begin{pmatrix} X-G^*G-D_GL^*D_{F^*} XD_{F^*}LD_G &-D_GL^*D_{F^*}XF
\cr -F^*XD_{F^*}LD_G & I_{\sD_D}- F^*X F
\end{pmatrix}\ge 0\\
 &0< X\le I_\sH
\end{split}
\right.,
\end{equation}
\begin{equation}
\label{SkypQX}
\left\{
\begin{split}
&(I_\sH-X)P_\sH\le \left(D^2_Q+Q^*(I_\sH-X)P'_\sH Q\right)_\sH\\
&0< X\le I_\sH
\end{split}
\right..
\end{equation}
\end{theorem}
\begin{proof} Note that \eqref{SkypQXX} are \eqref{CKYPQ}
written in terms of the entries. By Proposition \ref{EQ} the
inequalities \eqref{CKYP} and \eqref{CKYPQ} are equivalent. Let us
prove the equivalence of \eqref{CKYP} and \eqref{SkypX}. Suppose
that $X$ satisfies \eqref{CKYP} and put $Y=I_\sH-X.$ The operator
$Y$ belongs to the operator interval $[0,I_\sH]$ and
$\ker(I_\sH-Y)=\{0\}$. In terms of the operator $Y$ we have
\[
\begin{split}
&0\le\begin{pmatrix}{X}&0\cr 0&I_\sM\end{pmatrix}-T^*\begin{pmatrix}{X}&0\cr 0&I_\sN\end{pmatrix}=
\begin{pmatrix}I_\sH -Y&0\cr 0&I_\sM\end{pmatrix}-T^*\begin{pmatrix}I_\sH-Y&0\cr 0&I_\sN\end{pmatrix}T=\\
&=I-T^*T+T^*YP'_\sH T-YP_\sH,
\end{split}
\]
The weak form of the above inequality is the following
\begin{equation}
\label{wkyp} (Y x,x)\le \left(\left(D^2_T+T^*YP'_\sH
T\right)\begin{pmatrix}x\cr u\end{pmatrix},
\begin{pmatrix}x\cr u\end{pmatrix}\right),\; x\in\sH,\; u\in\sM.
\end{equation}
The equality \eqref{Sh1} for the shorted operator yields that the
operator $Y$ is a solution of the system
\begin{equation}
\label{Skyp}
\left\{
\begin{split}
&YP_\sH\le \left(D^2_T+T^*YP'_\sH T\right)_\sH,\\
&0\le Y<I_\sH
\end{split}
\right.
\end{equation}
If $X$ is a solution of the system \eqref{SkypX} then $Y=I_\sH-X$ satisfies \eqref{wkyp}and therefore $X$ satisfies \eqref{CKYP}.

Similarly \eqref{CKYPQ} is equivalent to \eqref{SkypQX}. Note that
by Proposition \ref{NEW} the right hand sides of \eqref{SkypX} and
\eqref{SkypQX} are equal. Using \eqref{EXSHORT} we get that
\eqref{SkypQX} is equivalent to \eqref{SHORTX}. By Sylvester's
criteria \eqref{CKYP1} is equivalent to \eqref{SHORTX}.
\end{proof}
\begin{proposition}
\label{optimal} Let the function $\Theta(\lambda)$ belongs to the
Schur class ${\bf S}(\sM,\sN)$ and let $Z(\lambda)$ be its M\"obius
parameter. Then the passive minimal realization
\[
\nu=\left\{\begin{pmatrix}S&F\cr G&0\end{pmatrix};\sH,\sD_{\Theta(0)},\sD_{\Theta^*(0)}\right\}
\]
 of $Z(\lambda)$ is optimal ((*)-optimal) if and only if the passive minimal realization
\[
\tau=\left\{\begin{pmatrix}-F\Theta^*(0)G+S&FD_{\Theta(0)}\cr D_{\Theta^*(0)}G&\Theta(0)\end{pmatrix};\sH,\sM,\sN\right\}
\]
of $\Theta(\lambda)$ is optimal ((*)-optimal).
\end{proposition}
\begin{proof} According to Theorem \ref{EQIN} the set of all solutions of the KYP inequality
\eqref{CKYP} for
$T=\begin{pmatrix}-F\Theta^*(0)G+S&FD_{\Theta(0)}\cr
D_{\Theta^*(0)}G&\Theta(0)\end{pmatrix}$ coincides with the set of
all solutions of the KYP inequality \eqref{CKYPQ} for
$Q=\begin{pmatrix}S&F\cr G&0\end{pmatrix}$. If the system $\nu$
($\tau$) is optimal realization of $Z(\lambda)$ ($\Theta(\lambda)$)
then the minimal solution of \eqref{CKYPQ} (\eqref{CKYP}) is
$X_0=I_\sH$. Therefore, the minimal solution of \eqref{CKYP}
(\eqref{CKYPQ}) is $X_0=I_\sH$ as well. Thus, the system $\tau$
($\nu$) is optimal realization of $\Theta(\lambda)$ ($Z(\lambda)$).
Passing to the adjoint systems
$$\nu^*=\left\{\begin{pmatrix}S^*&G^*\cr F^*&0\end{pmatrix};\sH,\sD_{\Theta^*(0)},\sD_{\Theta(0)}\right\}$$
and
$$\tau^*=\left\{\begin{pmatrix}-G^*\Theta(0)F^*+S^*&G^*D_{\Theta^*(0)}\cr D_{\Theta(0)}F^*&\Theta^*(0)\end{pmatrix};\sH,\sN, \sM\right\}$$
and their transfer functions
$Z^*(\overline \lambda)$ and $\Theta^*(\overline \lambda)$, respectively, we get  that
$\nu^*$ is (*)-optimal iff $\tau^*$ is (*)-optimal.
\end{proof}
The next theorem is an immediate consequence of Theorem \ref{EQIN}.
\begin{theorem}
\label{RICEQ} Let
\[ T=\begin{pmatrix} A&B \cr
C&D\end{pmatrix}=\begin{pmatrix}-FD^*G+D_{F^*}LD_G&FD_D\cr
D_{D^*}G&D\end{pmatrix}:
\begin{pmatrix}\sH\cr\sM\end{pmatrix}\to
\begin{pmatrix}\sH\cr\sN\end{pmatrix}
\]
 be a contraction and let
$Q:=\begin{pmatrix}D_{F^*}LD_G&F\cr
G&0\end{pmatrix}:\begin{pmatrix}\sH\cr\sD_D\end{pmatrix}\to\begin{pmatrix}\sH\cr
\sD_{D^*}\end{pmatrix}. $ Then the following equations are
equivalent on the operator interval $(0,I_\sH]$:
\begin{equation}
\label{RicXX}  X-A^*X A-C^*C-(A^*X B+
C^*D)(I_\sM-B^*XB-D^*D)^{-1}(B^*X A+D^* C)=0,
\end{equation}
\begin{equation}
\label{RicX} (I_\sH-X)P_\sH= \left(D^2_T+T^*(I_\sH-X)P'_\sH
T\right)_\sH,
\end{equation}
\begin{equation}
\label{RICQX} (I_\sH-X)P_\sH= \left(D^2_Q+Q^*(I_\sH-X)P'_\sH
Q\right)_\sH,
\end{equation}
\begin{equation}
\label{RicQXXX} X-G^*G-S^*XS-  S^*XF(I_\sH- F^*X F)^{-1}F^*XS= 0,
\end{equation}
\begin{equation}
\label{RICSHORTX}
 X=G^*G+S^*X^{1/2}(I_\sH-X^{1/2}FF^*X^{1/2})^{-1}X^{1/2}S,
\end{equation}
where $S=D_{F^*}LD_G$,
\[
\begin{split}
&S^*XF(I_\sH- F^*X
F)^{-1}F^*XS:=\\
&=\left((I_\sH-FXF^*)^{-1/2}F^*XS\right)^*\left((I_\sH-FXF^*)^{-1/2}F^*XS\right),
\end{split}
\]
and
\[
\begin{split}
&S^*X^{1/2}(I_\sH-X^{1/2}FF^*X^{1/2})^{-1}X^{1/2}S:=\\
&=D_GL^*\left((I_\sH-X^{1/2}FF^*X^{1/2})^{-1/2}X^{1/2}D_{F^*}\right)^*
\left((I_\sH-X^{1/2}FF^*X^{1/2})^{-1/2}X^{1/2}D_{F^*}\right)LD_G.
\end{split}
\]
Moreover, the equations \eqref{RicX} -- \eqref{RICSHORTX} are
equivalent to the equation \eqref{RIC1}.
\end{theorem}
 The equivalent equations \eqref{RicX} -- \eqref{RICSHORTX}
will be called the Riccati equations.
\begin{remark}
\label{RISO} Let $Q=\begin{pmatrix}S&F\cr
G&0\end{pmatrix}:\begin{pmatrix}\sH\cr\sL_1\end{pmatrix}\to
\begin{pmatrix}\sH\cr \sL_2\end{pmatrix}$
be an isometric operator. Then $F$ is isometry, $D_{F^*}$ is the
orthogonal projection, $S=D_{F^*}LD_G$, and  $D_L=0$. Denote
$\cK=\ker F^*=\ran D_{F^*}$. Since $D_{F^*}=P_\cK$. we get
  \[
 \begin{split}
 & D^2_{G}-D_GL^*D_{F^*}X^{1/2}(I_\sH-X^{1/2}FF^*X^{1/2})^{-1}X^{1/2}D_{F^*}LD_G=\\
&=S^*\left(P_\cK-P_{\cK}X^{1/2}(I_\sH-X^{1/2}P_{\cK^\perp}X^{1/2})^{-1}X^{1/2}P_{\cK}\right)S.
\end{split}
 \]
Taking into account Theorem \ref{newshort} and \eqref{SHORTX} we get
the corresponding KYP inequality
\[
I_\sH-X\le S^*\left(I_\sH-X\right)_{\cK}S.
\]
\end{remark}
\begin{example}
\label{EX2} Let $\sH$, $\sM$, and $\sN$ be separable Hilbert spaces.
Suppose that $G\in\bL(\sH,\sN)$ and $F\in\bL(\sM,\sH)$ are such that
$G^*G=FF^*=\alpha I_\sH$, where $\alpha\in(0,1)$. Then
$D_G=D_{F^*}=(1-\alpha)^{1/2}I_\sH$. Let $L$ be a unitary operator
in $\sH$. By Theorem \ref{ParContr} the operator
$$Q=\begin{pmatrix} (1-\alpha) L&F\cr
G&0\end{pmatrix}:\begin{pmatrix}\sH\cr\sM\end{pmatrix}\to\begin{pmatrix}\sH\cr\sN\end{pmatrix}$$
is a contraction.
 Consider the passive system
\[
\nu=\left\{\begin{pmatrix} (1-\alpha) L&F\cr G&0\end{pmatrix}
;\sH,\sM,\sN\right\}.
\]
Because $\ran F=\ran G^*=\sH$, the system $\nu$ is minimal.
 The corresponding Riccati equation \eqref{RICSHORTX} takes the
form
\begin{equation}
\label{SPETS} X=\alpha I_\sH+(1-\alpha)^2L^*X(I_\sH-\alpha
X)^{-1}L,\; 0< X\le  I_\sH.
\end{equation}
We will prove that this equation has a unique solution $X=I_\sH$.

Put $W=(1-\alpha)(I_\sH-\alpha X)^{-1}.$  Then $(1-\alpha)I_\sH<W\le
I_\sH.$  From \eqref{SPETS} we obtain the equation
\begin{equation}
\label{SPETS1} L^*WL+W^{-1}=2I_\sH,
\end{equation}
Clearly, \eqref{SPETS1} has a solution $W=I_\sH$. Let $W$ be any
solution of \eqref{SPETS1} such that $(1-\alpha)I_\sH<W\le I_\sH,$
i.e. $\sigma(W)\subset [1-\alpha, 1]$. Since $L$ is unitary
operator, from \eqref{SPETS1} it follows that
$\sigma(2I_\sH-W^{-1})=\sigma(W).$ Let $\lambda_0\in\sigma(W)$ then
$2-\lambda^{-1}_0\in\sigma(2I_\sH-W^{-1})=\sigma(W)$. Since
$2-\lambda^{-1}_0>0$, we get that
\[
\frac{1}{2}<\lambda_0\le 1.
\]
Because $\mu_0=2-\lambda^{-1}_0\in \sigma(W)$ and $1/2<\mu_0\le 1$,
we get
\[
\frac{2}{3}<\lambda_0\le 1.
\]
Thus
\[
\frac{2}{3}<2-\lambda^{-1}_0\le 1.
\]
It follows that
\[
\frac{3}{4}<\lambda_0\le 1.
\]
Continuing these reasonings, we get
\[
\frac{n}{n+1}<\lambda_0\le 1, \; n=1,2,\ldots.
\]
And now we get that $\lambda_0=1$, i.e. $\sigma(W)=\{1\}.$ Because
$W$ is selfadjoint operator we have $W=I_\sH$. Thus, the equation
\eqref{SPETS} has a unique solution $X=I_\sH.$
\end{example}

\section{Properties of solutions of the KYP inequality and Riccati equation}
\begin{proposition}
\label{observable} Let $\tau=\left\{\begin{pmatrix} A&B \cr
C&D\end{pmatrix};\sH,\sM,\sN\right\}$ be a passive observable
system. If a nonnegative contraction $X$ in $\sH$ is a solution of
the inequality
\begin{equation} \label{kernel}
  (I_\sH-X)P_\sH\le
\left(D^2_T+T^*(I_\sH-X)P'_\sH T\right)_\sH
\end{equation}
then $\ker X=\{0\}$.
\end{proposition}
\begin{proof}
Suppose that $X$ satisfies \eqref{kernel} and $\ker X\ne \{0\}$.
Then there is a nonzero vector $x$ in $\sH$ such that
$(I_\sH-X)x=x$. Since $D^2_T+T^*(I_\sH-X)P'_\sH T$ is a contraction,
we obtain $\left(D^2_T+T^*(I_\sH-X)P'_\sH T\right)_\sH=x$ and hence
$D^2_T+T^*(I_\sH-X)P'_\sH T=x$. It follows that
\[
P_\sN T x=0,\; X P'_\sH Tx=0.
\]
This means that $Cx=0$ and $XAx=0$. Replacing $x$ by $Ax$ we get
$CAx=0$ and $XA^2x=0$. By induction $CA^n x=0$ for all
$n=0,1,\ldots.$ Since the system $\tau$ is observable, we get $x=0$.
\end{proof}

\begin{theorem}
\label{T1} Let $T=\begin{pmatrix} A&B \cr
C&D\end{pmatrix}$:$\begin{pmatrix}\sH\cr \sM \end{pmatrix}$:$\to
\begin{pmatrix}\sH\cr \sN \end{pmatrix}$ be a contraction. Then every solution $Y$ of \eqref{Skyp} satisfies the estimate
\begin{equation}
\label{est1} Y\le\left(D^2_{P_\sN T}\right)_\sH\uphar\sH.
\end{equation}
If the system  $\tau=\left\{\begin{pmatrix} A&B \cr
C&D\end{pmatrix};\sH,\sM,\sN\right\}$ is observable then every $Y$
from the operator interval $[0,\left(D^2_T\right)_\sH\uphar \sH]$ is
a solution of \eqref{Skyp}.
\end{theorem}
\begin{proof} If $Y$ is a solution of \eqref{Skyp} then in view of
\[
D^2_T+T^*YP'_\sH T\le D^2_T+T^*P'_\sH T=I-T^*P_\sN T=D^2_{P_\sN T}
\]
we get $Y\le\left(D^2_{P_\sN T}\right)_\sH\uphar\sH.$

Suppose that the system $\tau$ is observable.
 By Proposition \ref{short} we have
\[
\left(D^2_T+T^*\left(D^2_T\right)_\sH P'_\sH T\right)_\sH\ge\left(D^2_T\right)_\sH P_\sH.
\]
It follows that the shorted operator $Y=\left(D^2_T\right)_\sH\uphar\sH$ is a solution of
the inequality
\[
YP_\sH\le \left(D^2_T+T^*YP'_\sH T\right)_\sH.
\]
Moreover, if $Y\in [0,\left(D^2_T\right)_\sH\uphar\sH]$ then
\[
YP_\sH\le \left(D^2_T\right)_\sH P_\sH\le \left(D^2_T+T^*YP'_\sH T\right)_\sH.
\]
By Proposition \ref{observable} we have $\ker (I_\sH-Y)=\{0\}.$
\end{proof}
\begin{corollary}
\label{Col1} Suppose that $\tau=\left\{\begin{pmatrix} A&B\cr
C&D\end{pmatrix},\sH,\sM,\sN\right\}$ is a passive observable
system. Then every $X$ from the operator interval
$[I_\sH-\left(D^2_T\right)_\sH\uphar \sH, I_\sH]$ is a solution of
\eqref{CKYP} and if $X$ is a solution of \eqref{CKYP} then $X\ge
I_\sH-\left(D^2_{P_\sN T}\right)_\sH\uphar\sH$.

Assume
\[
\left(D^2_{T}\right)_\sH=\left(D^2_{P_\sN T}\right)_\sH.
\]
Then the operator $X_{0}=I_\sH-\left(D^2_{P_\sN
T}\right)_\sH\uphar\sH$ is the minimal solution of KYP inequality
\eqref{CKYP}
\end{corollary}
\begin{corollary}
\label{Col2} Let $\Theta(\lambda)$ belongs to the Schur class ${\bf
S}(\sM,\sN)$. If the system
$$\tau=\left\{\begin{pmatrix} A&B\cr
C&D\end{pmatrix};\sH,\sM,\sN\right\}$$
 is a passive minimal and
optimal realization of $\Theta(\lambda)$ then
\begin{equation}
\label{nesopt}
(D^2_T)_\sH=0.
\end{equation}
\end{corollary}
\begin{proof}
If the system $\tau=\left\{\begin{pmatrix} A&B\cr
C&D\end{pmatrix};\sH,\sM,\sN\right\}$ is a passive minimal and
optimal realization of $\Theta(\lambda)$ then the unique solution
$X$ from the operator interval $[0,I_\sH]$ of the KYP inequality
\eqref{kyp1} is the identity operator $I_\sH$. If $(D^2_T)_\sH\ne 0$
then according to Corollary \ref{Col1} the operator
$I_\sH-(D^2_T)_\sH\uphar\sH$ is a solution of \eqref{kyp1}. It
follows that $(D^2_T)_\sH=0.$
\end{proof}
\begin{corollary}
\label{Col3} Let $\tau=\left\{\begin{pmatrix} A&B\cr
C&D\end{pmatrix};\sH,\sM,\sN\right\}$ be a passive minimal system.
Then the minimal $X_0$ solution of the KYP inequality \eqref{CKYP}
satisfies the Riccati equations \eqref{RicXX} -- \eqref{RICSHORTX}.
\end{corollary}
\begin{proof}
Since the system $\tau$ is passive and minimal, the minimal solution
$X_0$ of \eqref{kyp3} satisfies $0<X_0\le I_\sH$. Let the operator
$T_0$ be defined on the domain $\ran X_0^{1/2}\oplus\sM$ by the
equality
\[
T_0=\begin{pmatrix} X_0^{1/2}&0\cr 0&I_\sN\end{pmatrix}T\begin{pmatrix} X_0^{-1/2}&0\cr 0&I_\sM\end{pmatrix}.
\]
The operator $T_0$ is a contraction and has contractive continuation
on $\sH\oplus\sM$. We preserve the notation $T_0$ for this
continuation. The system $\tau_0=\{T_0,\sH,\sM,\sN\}$ is passive
minimal and optimal realization of the transfer function
$\Theta(\lambda)=D+\lambda C(I-\lambda A)^{-1}B$, $|\lambda|< 1$ for
the system $\tau$. According to Corollary \ref{Col2} the operator
$T_0$ satisfies the condition $(D^2_{T_0})_\sH=0.$ It follows that
\[
\inf\limits_{u\in\sM}\left\{\left\|D_{T_0}\begin{pmatrix}g\cr u\end{pmatrix}
\right\|^2\right\}=0
\]
for all $g\in\sH$. In particular
\[
\inf\limits_{u\in\sM}\left\{\left\|D_{T_0}\begin{pmatrix}X_0^{1/2}&0\cr 0&I_\sM\end{pmatrix} \begin{pmatrix}g\cr u\end{pmatrix}\right\|^2\right\}=0,\;g\in\sH.
\]
Since
\[
\begin{pmatrix} X_0&0\cr 0&I_\sM\end{pmatrix}-T^*\begin{pmatrix} X_0&0\cr 0&I_\sN\end{pmatrix}T=\begin{pmatrix} X_0^{1/2}&0\cr 0&I_\sM\end{pmatrix}D^2_{T_0}\begin{pmatrix} X_0^{1/2}&0\cr 0&I_\sM\end{pmatrix},
\]
we get
\[
\left(\begin{pmatrix} X_0&0\cr 0&I_\sM\end{pmatrix}-T^*\begin{pmatrix} X_0&0\cr 0&I_\sN\end{pmatrix}
T\right)_\sH=0.
\]
For $Y_0=I_\sH-X_0$ we have
\[
\begin{split}
&\begin{pmatrix}I_\sH -Y_0&0\cr 0&I_\sM\end{pmatrix}-T^*\begin{pmatrix}I_\sH-Y_0&0\cr 0&I_\sN\end{pmatrix}T=\\
&=D^2_{T}+T^*Y_0P'_\sH T-Y_0P_\sH.
\end{split}
\]
Since
\[
\left(\begin{pmatrix}I_\sH -Y_0&0\cr 0&I_\sM\end{pmatrix}-T^*\begin{pmatrix}I_\sH-Y_0&0\cr 0&I_\sN\end{pmatrix}
T\right)_\sH=0,
\]
we get
\[
0=(D^2_{T}+T^* Y_0P'_\sH T-Y_0P_\sH)_\sH=(D^2_{T}+T^* Y_0P'_\sH
T)_\sH- Y_0P_\sH.
\]
Thus, $Y_0$ satisfies the equation
\begin{equation}
\label{maxsol} YP_\sH=(I-T^*T+T^* YP'_\sH T)_\sH
\end{equation}
and $X_0=I_\sH-Y_0$ satisfies the equation \eqref{RicX}.
\end{proof}
\begin{corollary}
\label{Col4} Let $\Theta(\lambda)\in{\bf S}(\sM,\sN)$ and let
$\tau=\left\{\begin{pmatrix} A&B\cr
C&D\end{pmatrix};\sH,\sM,\sN\right\}$ be a passive minimal
realization of $\Theta$. If $(D^2_{P_\sN T})_\sH=0$ then the system
$\tau$ is optimal.
\end{corollary}
\begin{proof}
Since $D^2_T\le D^2_{P_\sN T}$ and $(D^2_{P_\sN T})_\sH=0$, we
obtain $(D^2_T)_\sH=0.$ By Corollary \ref{Col1} in this case the
minimal solution of \eqref{kyp1} is $X_0=I_\sH$. This means that
$\tau$ is the optimal realization of $\Theta$.
\end{proof}
\begin{remark}
\label{RRR} Let $\tau=\left\{\begin{pmatrix} A&B\cr
C&D\end{pmatrix};\sH,\sM,\sN\right\}$ be a minimal passive system
and let
\[
T=\begin{pmatrix}A&B\cr C&D\end{pmatrix}
 =\begin{pmatrix}-FD^*G+D_{F^*}LD_G&FD_D\cr
D_{D^*} G&D
\end{pmatrix}
\]
Then the statements of
 Corollaries \ref{Col1}--\ref{Col4} can be reformulated as follows:
\begin{enumerate}
\item every $X$ from the operator interval $[G^*G+D_GL^*LD_G, I]$ is
a solution of \eqref{CKYP1} and every solution of \eqref{CKYP1}
satisfies the estimate $X \ge G^*G$;
\item if $L=0$ then $X_0=G^*G$ is the minimal solution of
\eqref{CKYP1} (cf. \cite{ArKaaP4});
\item
if the system $\tau$ is optimal realization of the function
$\Theta(\lambda)\in{\bf S}(\sM,\sN)$ then $D_LD_G=0$;
\item if the system $\tau$ is a realization of the function
$\Theta(\lambda)\in{\bf S}(\sM,\sN)$ and if $G$ is isometry then the
system $\tau$ is optimal.
\end{enumerate}
\end{remark}
\begin{remark}
\label{RAN} Let $\tau=\left\{\begin{pmatrix} A&B\cr
C&D\end{pmatrix};\sH,\sM,\sN\right\}$ be a minimal system with
transfer function $\Theta(\lambda)$ from the Schur class ${\bf
S}(\sM,\sN)$. Suppose that the bounded positive selfadjoint operator
$X$ is such that the operator
\[
\delta(X)=I_\sM- D^*D-B^*XB
\]
is positive definite. Then the KYP inequality
\[
L(X)=
\begin{pmatrix} X-A^* XA-C^*C &
-A^*X B- C^*D\cr -B^*X A-D^* C & I- B^*X B-D^*D
\end{pmatrix}
\ge 0
\]
is equivalent to the inequality $R(X)\ge 0$, where
\[
R(X)=X-A^* XA-C^*C-B^*XA(I_\sM- D^*D-B^*XB)^{-1}A^*XB
\]
is the corresponding Schur complement. If there exists such $X$ that
$\delta(X)$ is positive definite and $R(X)\ge 0$ then for the
minimal solution $X_{\min}$ of KYP inequality we have
$\delta(X_{\min})\ge \delta(X)$ and $R(X_{\min})\ge 0$. For a finite
dimensional $\sH$ it was shown in \cite{RV} that the minimal
solution $X_{\min}$ satisfies the algebraic Riccati equation
$R(X_{\min})=0$. Thus, the statement of Corollary \ref{Col3} is the
generalization of the result in \cite{RV} for  a passive minimal
system with infinite dimensional state space.
\end{remark}
\begin{proposition}
\label{PPPP} Let $\Theta(\lambda)\in{\bf S}(\sM,\sN)$ and let the
M\"obius parameter $Z(\lambda)$ of $\Theta$ be of the form
$Z(\lambda)=\lambda K$, $K\in
\bL(\sD_{\Theta(0)},\sD_{\Theta^*(0)})$, $K\ne 0$. Then
\begin{enumerate}
\item
the minimal passive and optimal realization $\tau$ of $\Theta$ is
unitarily equivalent to the system
\[
\tau=\left\{\begin{pmatrix}-K\Theta^*(0)&KD_{\Theta(0)}\cr
D_{\Theta^*(0)}&\Theta(0)\end{pmatrix};\cran K,\sM,\sN  \right\};
\]
\item
the minimal passive and $(*)$- optimal realization $\tau$ of
$\Theta$ is unitarily equivalent to the system
\[
\eta=\left\{\begin{pmatrix}-P_{\cran K^*}\Theta^*(0)K\uphar\cran
K^*&P_{\cran K^*}D_{\Theta(0)}\cr D_{\Theta^*(0)}K\uphar\cran
K^*&\Theta(0)\end{pmatrix};\cran K^*,\sM,\sN \right\};
\]
\end{enumerate}
\end{proposition}
\begin{proof} Let $j$ be the embedding of $\cran K$ into $\sD_{\Theta^*(0)}$. Then the system
\[
\nu=\left\{\begin{pmatrix}0&K\cr j&0\end{pmatrix}; \cran
K,\sD_{\Theta(0)},\sD_{\Theta^*(0)}\right\}.
\]
is a passive and minimal realization of the function
$Z(\lambda)=\lambda K$ (see Proposition \ref{LIN}). The
corresponding Riccati equation \eqref{RicQXXX} takes the form
$X=I_{\cran K}.$ By Remark \ref{RRR}, the system $\nu$ is optimal
realization of $Z(\lambda)=\lambda K$.  From Proposition
\ref{optimal} it follows that the system
\[
\tau=\left\{\begin{pmatrix}-K\Theta^*(0)&KD_{\Theta(0)}\cr
D_{\Theta^*(0)}&\Theta(0)\end{pmatrix};\cran K,\sM,\sN  \right\}
\]
is minimal passive and optimal realization of $\Theta$.

The system
\[
\sigma=\left\{\begin{pmatrix}0&P_{\cran K^*}\cr K\uphar\cran
K^*&0\end{pmatrix};\cran
K^*,\sD_{\Theta(0)},\sD_{\Theta^*(0)}\right\}
\]
is the passive and minimal realization of the function $\lambda K$.
The KYP inequality \eqref{SHORTX} for the adjoint system $\sigma^*$
takes the form
\[
\left\{\begin{array}{l} X\ge I_{\cran K^*},\\ 0\le X\le I_{\cran
K^*}
 \end{array}
 \right..
\]
So, $X=I_{\cran K^*}$ is the minimal solution. It is the minimal
solution of the generalized KYP inequality for $\sigma^*$. Hence,
$X=I_{\cran K^*}$ is the maximal solution of the generalized KYP
inequality for $\sigma$. It follows that $\sigma$ is a $(*)$-optimal
realization of $Z(\lambda)=\lambda K$ and by Proposition
\ref{optimal} the system $\eta$ is a $(*)$-optimal realization of
$\Theta(\lambda).$
\end{proof}
The next theorems provides sufficient uniqueness conditions for the
solutions of the Riccati equation.
\begin{theorem}
\label{intersec} Let $T=\begin{pmatrix} A&B\cr
C&D\end{pmatrix}$:$\begin{pmatrix}\sH\cr\sM
\end{pmatrix}$ $\to\begin{pmatrix}\sH\cr\sN
\end{pmatrix}$ be a contraction. Suppose
that
\begin{equation}
\label{UNIQQ} \left\{
\begin{split}
&\left(D^2_T\right)_\sH=0,\\
&\ran\left(\left(D^2_{P_\sN T}\right)_\sH\right)^{1/2} \bigcap
\ran\left(\left(D^2_{P_\sM T^*}\right)_\sH\right)^{1/2}\subset\ran
\left(\left(D^2_{T^*}\right)_\sH\right)^{1/2}.
\end{split}
\right.
\end{equation}
Then the Riccati \eqref{RicX} equation has a unique solution
$X=I_\sH$.
\end{theorem}
\begin{proof} Let $T$ takes the form
\[
T=\begin{pmatrix}-FD^*G+D_{F^*}LD_G&FD_D\cr D_{D^*} G&D
\end{pmatrix}
\]
with contractions $D,$ $F,$ $G$, and $L$.  From \eqref{defshort} it
follows that if $\left(D^2_{P_\sN T}\right)_\sH=0$ then $X=I_\sH$ is
a unique solution of \eqref{RicX}. Assume $\left(D^2_{P_\sN
T}\right)_\sH\ne 0$. Since $(D^2_T)_\sH=0$, from the equivalences
\eqref{RAV} it follows that
\[
\left(D^2_{P_\sM T^*}\right)_\sH\ne 0.
\]
From \eqref{UNIQQ} and \eqref{defshort} we have  $D_G\ne 0,$
$D_{F^*}\ne 0$, $D_L=0,$ and
\begin{equation}
\label{IINTER} \ran D_G\cap\ran D_{F^*}\subset\ran (D_{F^*}D_{L^*}).
\end{equation}
According to Proposition \ref{RICEQ} the equation \eqref{RicX} is
equivalent to the equation \eqref{RICSHORTX}. We will prove that
\eqref{RICSHORTX} has a unique solution $X=I_\sH.$ Suppose that $X$
is a solution.  Define  $\Psi:=I_\sH- X^{1/2}FF^*X^{1/2}$. Since
$\Psi=I_\sH-X+X^{1/2}D^2_{F^*}X^{1/2}$, we have $\Psi\ge I_\sH-X$
and $\Psi\ge X^{1/2}D^2_{F^*}X^{1/2}$. Therefore
\[
(I_\sH-X)^{1/2}=U\Psi^{1/2},\; D_{F^*}X^{1/2}=V\Psi^{1/2},
\]
where $U:\cran \Psi^{1/2}\to \cran (I_\sH-X)^{1/2}$, $V:\cran
\Psi^{1/2}\to \cran D_{F^*}=\cran (D_{F^*}X^{1/2})$, and
$U^*U+V^*V=I_{\cran \Psi^{1/2}}$. Hence $U^*U=D^2_{V}$. Since
$X^{1/2}D_{F^*}=\Psi^{1/2}V^*$, we get
\[
X^{1/2}D_{F^*}D_{V^*}=\Psi^{1/2}V^*D_{V^*}=\Psi^{1/2}D_{V}V^*.
\]
From
\[
I_\sH-X=\Psi^{1/2}U^*U\Psi^{1/2}=\Psi^{1/2}D^2_V\Psi^{1/2}
\]
we get that $\ran(I_\sH-X)^{1/2}=\Psi^{1/2}\ran D_V.$
Therefore,
\begin{equation}
\label{inclus}
\ran X^{1/2}D_{F^*}D_{V^*}\subset \ran(I_\sH-X)^{1/2}.
\end{equation}
Using the well known relation
$$\ran X^{1/2}\cap\ran (I_\sH-X)^{1/2}=\ran
(X^{1/2}(I_\sH-X)^{1/2})$$ for every $X\in [0, I_\sH]$, from
\eqref{inclus} we get that
$$D_{F^*}\ran
D_{V^*}\subset\ran(I_\sH-X)^{1/2}.$$
The equation \eqref{RICSHORTX}
can be rewritten as follows
\[
X=G^*G+D_GL^*VV^*LD_G.
\]
Since $L^*L=I_{\sD_G}$, we get $I_\sH-X=D_GL^*D^2_{V^*}LD_G.$ It
follows that
\[
\ran (I_\sH-X)^{1/2}=D_GL^*\ran D_{V^*}\subset\ran D_G.
\]
Now we obtain
\[
 D_{F^*}\ran D_{V^*}\subset\ran D_G\cap\ran
D_{F^*}\subset D_{F^*}\ran D_{L^*}.
\]
Hence $\ran D_{V^*}\subset\ran D_{L^*}$. Since
$L:\sD_{G}\to\sD_{F^*}$ is isometry, we get $\ker L^*=\ran D_{L^*}$.
Therefore $L^*\uphar\ran D_{V^*}=0$. It follows $\ran
(I_\sH-X)^{1/2}=\{0\}$, i.e., $X=I_\sH$.
\end{proof}
Observe, Example \ref{EX2} shows that conditions \eqref{UNIQQ} are
not necessary for the uniqueness of the solutions of the KYP
inequality \eqref{SkypX}.
\begin{theorem}
\label{unique1} Let a contraction $T=\begin{pmatrix} A&B\cr
C&D\end{pmatrix}$:$\begin{pmatrix}\sH\cr\sM
\end{pmatrix}$ $\to\begin{pmatrix}\sH\cr\sN
\end{pmatrix}$ possesses the properties
\begin{equation}
\label{UNIQ1} \left\{
\begin{split}
&\left(D^2_T\right)_\sH=0,\;\left(D^2_{T^*}\right)_\sH= 0,\\
&\ran\left(\left(D^2_{P_\sN T}\right)_\sH\right)^{1/2} \bigcap
\ran\left(\left(D^2_{P_\sM T^*}\right)_\sH\right)^{1/2}=\{0\}
\end{split}
\right..
\end{equation}
Then the generalized KYP inequality \eqref{kyp3} has a unique
solution $X=I_\sH$.
\end{theorem}
\begin{proof}
Let  $T^*=\begin{pmatrix} A^*&C^*\cr
B^*&D^*\end{pmatrix}$:$\begin{pmatrix}\sH\cr\sN
\end{pmatrix}$ $\to\begin{pmatrix}\sH\cr\sM
\end{pmatrix}$ be the adjoint operator and let
\begin{equation}
\label{SkypX*} \left\{
\begin{split}
&(I_\sH-Z)P_\sH\le \left(D^2_{T^*}+T(I_\sH-Z)P_\sH T^*\right)_\sH\\
&0< Z\le I_\sH
\end{split}
\right.\;,
\end{equation}
\begin{equation}
\label{RicX*} \left\{\begin{array}{l}(I_\sH-Z)P_\sH=
\left(D^2_{T^*}+T(I_\sH-Z)P_\sH T^*\right)_\sH,\\
0<Z\le I_\sH
\end{array}
\right.
\end{equation}
be the corresponding KYP inequality and Riccati equation. By Theorem
\ref{intersec} the identity operator $I_\sH$ is a unique solution of
the Riccati equations \eqref{RicX} and \eqref{RicX*}.

Let us show that the passive system
\[
\tau=\left\{\begin{pmatrix} A&B\cr
C&D\end{pmatrix};\sH,\sM,\sN\right\}
\]
is minimal. Consider the parametrization of the contraction $T$:
\[
T=\begin{pmatrix}-FD^*G+D_{F^*}LD_G&FCD_D\cr D_{D^*} G&D
\end{pmatrix}
\]
and let $Q=\begin{pmatrix}D_{F^*}LD_G&F\cr
G&0\end{pmatrix}:\begin{pmatrix}\sH\cr\sD_D\end{pmatrix}\to
\begin{pmatrix}\sH\cr \sD_{D^*}\end{pmatrix}.$

Assume $\left(D^2_{P_\sN T}\right)_\sH\ne 0$ and $\left(D^2_{P_\sM
T^*}\right)_\sH\ne 0.$ From \eqref{UNIQ1} and \eqref{defshort} we
get the equalities
\[
D_LD_G=D_{L^*}D_{F^*}=0,\; \ran D_G\cap\ran D_{F^*}=\{0\}.
\]
It is known that
\[
\ran F+\ran D_{F^*}=\sH.
\]
Suppose that $F^*D_GL^*D_{F^*}f=0,$ where $f\in\sD_{F^*}$. Then the
vector $D_GL^*D_{F^*}f\in\ker F^*$, hence $D_GL^*D_{F^*}f\in\ran
D_{F^*}$. Since $L:\sD_G\to\sD_{F^*}$ is unitary operator and $ \ran
D_G\cap\ran D_{F^*}=\{0\}$, we get $f=0$. It follows that $
\cran(D_{F^*}LD_G F)=\sD_{F^*}$ and
\[
\cspan\{(D_{F^*}LD_G )^nF\sM,\; n=0,1,2\ldots\}=\sH.
\]
Thus, the system $\nu=\left\{\begin{pmatrix}D_{F^*}LD_G& F\cr
G&0\end{pmatrix};\sH,\sD_{D},\sD_{D^*}\right\}$ is controllable.
Similarly, the system $\nu$ is observable. By Theorem \eqref{TT1}
the system $\tau$ is minimal.

If $\left(D^2_{P_\sN T}\right)_\sH=D_G=0$ then by \eqref{RAV} also
$\left(D^2_{P_\sM T^*}\right)_\sH=D_{F^*}=0.$ Therefore $\ran F=\ran
G^*=\sH$ and $Q=\begin{pmatrix}0&F\cr G&0\end{pmatrix}$. It follows
that  in this case the systems $\nu$ and $\tau$ are minimal.
According to the result of \cite{ArKaaP3} the KYP inequality
\eqref{SkypX} has a minimal solution. Since $I_\sH$ is a solution of
\eqref{SkypX} and the Riccati equation \eqref{RicX} has a unique
solution $X=I_\sH$, the inequality \eqref{SkypX} has a unique
solution $X=I_\sH.$ Similarly the KYP inequality \eqref{SkypX*} also
has a unique solution $Z=I_\sH$. Hence, the minimal solution of the
generalized KYP inequality \eqref{kyp3}
\[
\begin{split}
&\left\|\begin{pmatrix}{X}^{1/2}&0\cr
0&I_\sM\end{pmatrix}\begin{pmatrix}x\cr u\end{pmatrix}
\right\|^2-\left\|\begin{pmatrix}{X}^{1/2}&0\cr
0&I_\sN\end{pmatrix}\begin{pmatrix} A&B
\cr C&D\end{pmatrix}\begin{pmatrix}x\cr u\end{pmatrix}\right\|^2\ge 0\\
& \qquad\mbox{for all}\quad\; x\in\dom X^{1/2},\, u\in\sM
\end{split}
\]
is $X=I_\sH$. Since  $Z$ is a solution of the generalized KYP
inequality for the adjoint operator \eqref{kyp3*} $\iff$ $X=Z^{-1}$
is a solution $X$ of \eqref{kyp3}, the minimal solution of
\eqref{kyp3*} is $Z=I_\sH$. Hence, the identity operator $I_\sH$ is
also the maximal solution of \eqref{kyp3}. So, \eqref{kyp3} has a
unique solution $X=I_\sH$.
\end{proof}
\begin{corollary}
Let $ \tau=\left\{\begin{pmatrix} A&B\cr
C&D\end{pmatrix};\sH,\sM,\sN\right\} $ be a passive system and let
$\Theta(\lambda)$ be its transfer function. If the operator
$T=\begin{pmatrix} A&B\cr
C&D\end{pmatrix}$:$\begin{pmatrix}\sH\cr\sM
\end{pmatrix}$ $\to\begin{pmatrix}\sH\cr\sN
\end{pmatrix}$ possesses the properties
\eqref{UNIQ1} then all minimal passive realizations of $\Theta$ are
unitary equivalent.
\end{corollary}
\begin{remark}
The conditions \eqref{UNIQQ} are equivalent to the following:
\[
\left\{\begin{array}{l} \ran D_T\cap\sH=\{0\},\\
\left(\ran D_{P_\sN T}\cap \sH\right)\bigcap\left(\ran D_{P_\sM
T^*}\cap \sH\right)\subset\ran D_{T^*}\cap\sH,
\end{array}
\right.
\]
and
\[
\eqref{UNIQ1} \iff
\left\{\begin{array}{l} \ran D_T\cap\sH=\ran D_{T^*}\cap\sH=\{0\},\\
\left(\ran D_{P_\sN T}\cap \sH\right)\bigcap\left(\ran D_{P_\sM
T^*}\cap \sH\right)=\{0\}.
\end{array}
\right.
\]
\end{remark}
\section {Approximation of the minimal solution} \label{approx}
The solutions of the Riccati equations \eqref{RicXX}--
\eqref{RICSHORTX} are fixed points of the corresponding maps. We
will prove that extremal solutions can be obtained by iteration
procedures with a special initial points.
\begin{theorem}
\label{iter0} Let $\tau=\left\{\begin{pmatrix} A&B\cr
C&D\end{pmatrix};\sH,\sM,\sN\right\}$ be a passive observable system
and let $T=\begin{pmatrix} A&B\cr C&D\end{pmatrix}$. Let us define
the sequence of nonnegative contractions in $\sH$:
\begin{equation}
\label{seq0} Y^{(0)}:=I_\sH,\; Y^{(n+1)}:= \left(D^2_T
+T^*Y^{(n)}P'_\sH T\right)_\sH\uphar\sH,\; n=0,1,\ldots.
\end{equation}
Then
\begin{enumerate}
\item the sequence $\{Y^{(n)}\}_{n=0}^\infty$ is a nonincreasing,
\item the operator
$$Y_0:=s-\lim\limits_{n\to \infty}Y^{(n)}$$
satisfies the equality
\begin{equation}
\label{yo0}
Y_0P_\sH=\left(D^2_T +T^*Y_0P'_\sH T\right)_\sH
\end{equation}
and $\ker (I_\sH-Y_0)=\{0\}$,
\item the operator $Y_0$ is a maximal solution of \eqref{Skyp}.
\end{enumerate}
\end{theorem}
\begin{proof}
Let us show that the sequence defined by \eqref{seq0} is
nonincreasing. Since $(D^2_{P_\sN T})_\sH\le P_\sH$, we get
\[
Y^{(1)}P_\sH=\left(D^2_T +T^*P'_\sH T\right)_\sH= \left(D^2_{P_\sN
T}\right)_\sH.
\]
Hence $Y^{(1)}\le Y^{(0)}$. Suppose that $Y^{(n)}\le Y^{(n-1)}$ for
given $n\ge 1$. Then
\[
Y^{(n+1)}P_\sH=\left(D^2_T +T^*Y^{(n)}P'_\sH T\right)_\sH\le \left(D^2_T +T^*Y^{(n-1)}P'_\sH T\right)_\sH=Y^{(n)}P_\sH.
\]
Thus, the sequence $\{Y^{(n)}\}_{n=0}^\infty$ is nonincreasing.
Because the operators $Y^{(n)}$ are nonnegative, there exists a
strong limit
$$Y_0=s-\lim\limits_{n\to \infty}Y^{(n)}.$$
Since
$$ Y^{(n+1)}=\left(D^2_T +T^*Y^{(n)}P'_\sH T\right)_\sH\uphar\sH,\; n=0,1,\ldots,$$
applying Proposition \ref{short} we get \eqref{yo0}.

Let us show that every solution $Y$ of \eqref{Skyp} satisfies the
inequality $Y\le Y_0$.  Suppose that $Y$ is a solution of
\eqref{Skyp}. Taking into account \eqref{est1} we get $Y\le
Y^{(1)}$. If it is proved that $Y\le Y^{(n)}$ for some $n\ge 1$ then
\[
Y\le \left(D^2_T+ T^*YP'_\sH T\right)_\sH\uphar\sH
\le \left(D^2_T+ T^*Y^{(n)} P'_\sH T\right)_\sH\uphar\sH=Y^{(n+1)}.
\]
By induction it follows that $Y\le Y_0.$ Using Proposition
\ref{observable} we get $\ker(I-Y_0)=\{0\}.$
 \end{proof}

\begin{remark}
\label{rem11} The  nondecreasing sequence
\[
X^{(0)}=I_\sH-Y^{(0)}=0,\;
X^{(n+1)}=I_\sH-Y^{(n+1)}=I_\sH-\left(D^2_T
+T^*(I_\sH-X^{(n)})P'_\sH T\right)_\sH\uphar\sH\
\]
strongly converges to the minimal solution $X_0$ of the KYP
inequality \eqref{SkypX} and the Riccati equations \eqref{RicXX}--
\eqref{RICSHORTX} (see Theorem \ref{RICEQ}). From \eqref{shortmap}
and \eqref{EXSHORT} we get
\[
\begin{split}
&X^{(n+1)}=G^*G+\\
&\qquad+
D_GL^*D_{F^*}{(X^{(n)})}^{1/2}\left(I_\sH-(X^{(n)})^{1/2}FF^*(X^{(n)})^{1/2}\right)^{-1}
(X^{(n)})^{1/2}D_{F^*}LD_G.
\end{split}
\]
\end{remark}
\begin{example}
 \label{EX1}
 Let $F\in\bL(\sM,\sH)$ be
a strict contraction ($||Fh||_\sH<||h||_\sM$ for all
$h\in\sM\setminus\{0\}$) and $\ker F^*=\{0\}$. Then $\cran
D_{F^*}=\sH$. Let $\alpha\in(0,1)$ and suppose that the operator
$G\in\bL(\sH,\sN)$ is chosen such that $G^*G=\alpha FF^*$. Then
$$\ker G=\{0\},\; D_G=\left(I_\sH-\alpha
FF^*\right)^{1/2}$$ and $\ran D_{G}=\sH$. Therefore $\ran
D_{F^*}\subset\ran D_{G}.$ Let $L=I_\sH$. By Theorem \ref{ParContr}
the operator
 \[
Q=\begin{pmatrix} D_{F^*}D_G&F\cr
G&0\end{pmatrix}:\begin{pmatrix}\sH\cr
\sM\end{pmatrix}\to\begin{pmatrix}\sH\cr \sN\end{pmatrix}
\]
is a contraction and from \eqref{defshort} we get that
$\left(D^2_Q\right)_\sH=0,$ $\left(D^2_{P_\sN Q}\right)_\sH\ne 0,$
and
$$\ran\left(\left(D^2_{P_\sM Q^*}\right)_\sH\right)^{1/2}\subset\ran\left(\left(D^2_{P_\sN Q}\right)_\sH\right)^{1/2}.$$
The system
\[
\nu =\left\{\begin{pmatrix} D_{F^*}D_G&F\cr
G&0\end{pmatrix};\sH,\sM,\sN\right\}
\]
is passive. The condition $\ker F^*=\{0\}$ yields that
\[
\bigcap_{n\ge 0}\ker\left(F^*(D_G D_{F^*})^n\right)=\bigcap_{n\ge
0}\ker\left(G(D_{F^*}D_G)^n\right)=\{0\}.
\]
So, the system $\nu$ is minimal. Its transfer function $Z(\lambda)$
takes the form
\[
Z(\lambda)=\lambda
G\left(I_\sH-\lambda(I_\sH-FF^*)^{1/2}(I_\sH-\alpha
FF^*)^{1/2}\right)^{-1}F.
\]
The corresponding Riccati equation  \eqref{RICSHORTX} takes the form
\begin{equation}
\label{spec1}
\left\{
\begin{split}
&X=\alpha FF^*+(I_\sH-\alpha
FF^*)^{1/2}D_{F^*}X^{1/2}(I_\sH-X^{1/2}FF^*X^{1/2})^{-1}X^{1/2}D_{F^*}(I_\sH-\alpha
FF^*)^{1/2}\\
&0<X\le I_\sH
\end{split}
\right.
\end{equation}
and has a solution $X_0=\alpha I_\sH$. Because $\alpha I_\sH<
I_\sH$, the system $\nu$ is non-optimal realization of $Z(\lambda)$.

Let us show that $X_0=\alpha I_\sH$ is the minimal solution of
\eqref{spec1}. Note that $X_{\min}\le X_0=\alpha I_\sH$. According
to Remark \ref{rem11} the sequence of operators
\[
\begin{split}
& X^{(0)}=0,\;X^{(n+1)}=\alpha FF^*+\\
&+(I_\sH-\alpha
FF^*)^{1/2}D_{F^*}(X^{(n)})^{1/2}(I_\sH-(X^{(n)})^{1/2}FF^*(X^{(n)})^{1/2})^{-1}(X^{(n)})^{1/2}D_{F^*}(I_\sH-\alpha
FF^*)^{1/2},\\
&\qquad n=0,1,\ldots
\end{split}
\]
is nondecreasing and strongly converges to the minimal solution
$X_{\min}$ of \eqref{spec1}. Hence $X^{(n)}\le \alpha I_\sH$ and
because $X^{(1)}=\alpha FF^*$, one has $X^{(n)}FF^*=FF^* X^{(n)}$
for all $n$. It follows that $X_{\min}FF^*=FF^*X_{\min}$ and
\[
\begin{split}
&I_\sH-X_{\min}=(I_\sH -\alpha
FF^*)\left(I_\sH-D^2_{F^*}X_{\min}(I_\sH-X_{\min}FF^*)^{-1}\right)=\\
&=(I-X_{\min})(I_\sH -\alpha FF^*)(I_\sH-X_{\min}FF^*)^{-1}.
\end{split}
\]
Hence
\[
(I_\sH-X_{\min})(\alpha I_\sH - X_{\min})FF^*=0.
\]
Therefore $(\alpha I_\sH-X_{\min})FF^*=0.$ Taking into account that
$\ker F^*=\{0\}$, we get $X_{\min}=\alpha I_\sH$.

Note that if the orthogonal projection $P$ in $\sH$ commutes with
$FF^*$ then the operator $X=P+\alpha P^\perp$ is a solution of the
Riccati equation \eqref{spec1}.

 Consider the adjoint system
\[
\nu^* =\left\{\begin{pmatrix} D_GD_{F^*}&G^*\cr
F^*&0\end{pmatrix};\sH,\sN,\sM\right\}.
\]
We will show that $X=I_\sH$ is the minimal solution of the
corresponding Riccati equation
\[
\left\{
\begin{split}
&X=FF^*+D_{F^*}(I_\sH-\alpha FF^*)^{1/2}X^{1/2}(I_\sH-\alpha
FF^*)^{-1}X^{1/2}(I_\sH-\alpha X^{1/2}FF^*X^{1/2})^{1/2}D_{F^*}\\
&0<X\le I_\sH
\end{split}
\right. \;.
\]
According to Remark \ref{rem11} the sequence of operators
\[
\begin{split}
& X^{(0)}=0,\;X^{(n+1)}= FF^*+\\
&+D_{F^*}(I_\sH-\alpha
FF^*)^{1/2}(X^{(n)})^{1/2}(I_\sH-\alpha(X^{(n)})^{1/2}FF^*(X^{(n)})^{1/2})^{-1}(X^{(n)})^{1/2}
(I_\sH-\alpha
FF^*)^{1/2}D_{F^*},\\
&\qquad n=0,1,\ldots
\end{split}
\]
is nondecreasing and strongly converges to the minimal solution
$X_{\min}$. It follows  that $X^{(n)}FF^*=FF^* X^{(n)}$ for all $n$,
 $X_{\min}FF^*=FF^*X_{\min}$ and
\[
\begin{split}
&I_\sH-X_{\min}=(I_\sH -FF^*)\left(I_\sH-(I_\sH -\alpha F{F^*})X_{\min}(I_\sH-\alpha X_{\min}FF^*)^{-1}\right)=\\
&=(I_\sH-X_{\min})(I_\sH -FF^*)(I_\sH -\alpha FF^*X_{\min})^{-1}.
\end{split}
\]
Hence
\[
(I_\sH-X_{\min})( I_\sH-\alpha X_{\min})=0.
\]
 Because $I_\sH -\alpha X_{\min}$ has bounded inverse, we get
$X_{\min}=I_\sH.$ Thus, the minimal passive system $\nu$ is
$(*)$-optimal realization of the function $Z(\lambda)$.
\end{example}
Observe that this example shows that the condition \eqref{nesopt}
for a passive minimal system is not sufficient for optimality.

\end{document}